\documentclass[reqno,a4paper]{amsart}
\usepackage[draft=false]{hyperref}
\usepackage{microtype}
\usepackage{url}
\usepackage{amsmath,amssymb,amsfonts}
   \DeclareMathOperator{\Id}{Id}
   \DeclareMathOperator{\e}{e}
   \DeclareMathOperator{\Lip}{Lip}
   \DeclareMathOperator{\sgn}{sgn}
\usepackage{amsthm}
   \theoremstyle{plain}
      \newtheorem{theorem}{Theorem}[section]
      \newtheorem{lemma}[theorem]{Lemma}
      \newtheorem{corollary}[theorem]{Corollary}
   \theoremstyle{definition}
      \newtheorem{example}[theorem]{Example}
      
      \newtheorem{remarks}[theorem]{Remarks}
\usepackage{euscript,dsfont}
   \newcommand{\R}{\mathds{R}}
   \newcommand{\cB}{\mathcal{B}_M}
   \newcommand{\cX}{\mathcal{A}}
   \newcommand{\cC}{\mathcal{C}_{M,N}}
   \newcommand{\cV}{\mathcal{V}}
   \newcommand{\ga}{\mathfrak{a}}
   \newcommand{\gb}{\mathfrak{b}}
   \newcommand{\gc}{\mathfrak{c}}
   \newcommand{\gd}{\mathfrak{d}}
   \newcommand{\nabcd}{\prts{\ga,\gb,\gc,\gd}-}
   \newcommand{\feps}{\mathfrak{\eps}}
   \newcommand{\epsa}{\feps_\ga}
   \newcommand{\epsb}{\feps_\gb}
   \newcommand{\epsc}{\feps_\gc}
   \newcommand{\epsd}{\feps_\gd}
   \newcommand{\epsi}{\feps_i}
\usepackage{enumitem}
   \newcommand{\lb}{label}
   \newcommand{\lm}{leftmargin}
   \newcommand{\ts}{topsep}

\renewcommand{\phi}{\varphi}
\newcommand{\eps}{\varepsilon}
\newcommand{\lbd}{\lambda}
\newcommand{\omg}{\omega}
\newcommand{\sgm}{\sigma}
\newcommand{\bxi}{\,\overline{\!\xi}}
\newcommand{\Lipfr}{\Lip\prts{f_r}}
\newcommand{\Lipft}{\Lip\prts{f_t}}
\newcommand{\prts}[1]{\left(#1\right)}
\newcommand{\prtsr}[1]{\left[#1\right]}
\newcommand{\abs}[1]{\left|#1\right|}
\newcommand{\absp}{\abs{p}}
\newcommand{\norm}[1]{\left\|#1\right\|}
\newcommand{\set}[1]{\left\{#1\right\}}
\newcommand{\setm}[1]{\setminus\set{#1}}
\newcommand{\maxs}[1]{\max\set{#1}}
\newcommand{\mins}[1]{\min\set{#1}}
\newcommand{\sups}[1]{\sup\set{#1}}
\newcommand{\pfrac}[2]{\prts{\dfrac{#1}{#2}}}
\newcommand{\prfrac}[2]{\prtsr{\dfrac{#1}{#2}}}
\newcommand{\dint}{\displaystyle\int}
\newcommand{\dlim}{\displaystyle\lim}
\newcommand{\dr}{\, dr}
\newcommand{\dtau}{\, d\tau}
\renewcommand{\ge}{\geqslant}

\renewcommand{\le}{\leqslant}

\renewcommand{\subset}{\subseteq}
\begin{document}
\title
   [Global Lipschitz Invariant Center Manifolds for ODE\lowercase{s} ...]%
   {Global Lipschitz Invariant Center Manifolds for ODE\lowercase{s} with Generalized Trichotomies}
\author[Ant\'onio J. G. Bento]{Ant\'onio J. G. Bento}
   \address{
      Ant\'onio J. G. Bento\\
      Departamento de Matem\'atica\\
      Universidade da Beira Interior\\
      6201-001 Covilh\~a\\
      Portugal}
   \address{
      Centro de Matem\'atica e Aplica\c{c}\~oes (CMA-UBI)\\
      Universidade da Beira Interior\\
      6201-001 Covilh\~a\\
      Portugal}
   \email{bento@ubi.pt}
\author[Cristina Tom\'as da Costa]{Cristina Tom\'as da Costa}
\address{
   Cristina Tom\'as da Costa\\
   \'Area Cient\'ifica de Matem\'atica\\
   Instituto Polit\'ecnico de Viseu\\
   Campus Polit\'ecnico Repeses\\
   3504-510 Viseu\\
   Portugal}
\email{ccosta@mat.estv.ipv.pt}
\date{\today}
\begin{abstract}
   In a Banach space, assuming that a linear nonautonomous differential equation $v'=A(t)v$ admits a very general type of trichotomy, we establish conditions for the existence of global Lipschitz invariant center manifold of the perturbed equation $v'=A(t)v+f(t,v)$. Our results not only improve results already existing in the literature, as well include new cases.
\end{abstract}
\maketitle
\section{Introduction}
Let $X$ be a Banach space and let $B(X)$ be the Banach algebra of all bounded linear operators acting on $X$. In this paper we are going to study the existence of global Lipschitz invariant center manifolds for differential equations of the type
\begin{equation*}
   v' = A(t) v + f(t,v),
\end{equation*}
where $A \colon \R \to B(X)$ is continuous, the perturbation $f \colon \R \times X \to X$ is a continuous function such that $f(t,0) = 0$ for every $t \in \R$, the function $f_t \colon X \to X$ given by $f_t (v) = f(t,v)$ is Lipschitz for every $t \in \R$ and the linear differential equation
\begin{equation*}
   v' = A(t) v
\end{equation*}
admits a very general type of trichotomy.

Center manifolds are a powerful tool in the study of stability and in the study of bifurcations because in many cases allow the reduction of the dimension of the state space (see Carr~\cite{Carr-book-1981}, Henry~\cite{Henry-book-1981}, Guckenheimer and Holmes~\cite{Guckenheimer_Holmes-book-1990}, Hale and Ko\c cak~\cite{Hale_Kocak-book-1991} and Haragus and Iooss~\cite{Haragus_Iooss-book-2011}). The first results on the existence of center manifolds were obtained by Pliss~\cite{Pliss-IANSSM-1964} in 1964 and by Kelley~\cite{Kelley-JMAA-1967,Kelley-JDE-1967} in 1967. After that many authors studied the problem and proved results about center manifolds. A good expository paper for the case of autonomous differential equations in finite dimension was written by Vanderbauwhede~\cite{Vanderbauwhede-DR-1989} (see also Vanderbauwhede and Gils~\cite{Vanderbauwhede_Gils-JFA-1987}) and for the case of autonomous differential equations in infinite dimension we recommend Vanderbauwhede and Iooss~\cite{Vanderbauwhede_Iooss-DR-1992}. For more details in the finite dimensional case see Chow, Liu and Yi~\cite{Chow_Liu_Yi-TAMS-2000,Chow_Liu_Yi-JDE-2000} and for the infinite dimensional case see Sijbrand~\cite{Sijbrand-TAMS-1985}, Mielke~\cite{Mielke-JDE-1986}, Chow and Lu~\cite{Chow_Lu-PRSE-1988,Chow_Lu-JDE-1988} and Chicone and Latushkin~\cite{Chicone_Latushkin-JDE-1997}.

For nonautonomous differential equations the concept of exponential trichotomy is an important tool to obtain center manifolds theorems. This notion goes back to Sacker and Sell~\cite{Sacker_Sell-JDE-1976-22-(497-522)}, Aulbach~\cite{Aulbach-NATMA-1982} and Elaydi and Hajek~\cite{Elaydi_Hajek-JMAA-1988} and is inspired by the notion of exponential dichotomy that can be traced back to the work of Perron in~\cite{Perron-MZ-1929,Perron-MZ-1930}. However, as in the case of exponential dichotomies, the notion of exponential trichotomy is very demanding and several generalizations have appeared in the literature. Essentially we can find two ways of generalization: on one hand replace the exponential growth rates by nonexponential growth rates and on the other hand consider exponential trichotomies that also depend on the initial time and hence are nonuniform. Trichotomies with nonexponential growth rates have been introduced by Fenner and Pinto in~\cite{Fenner_Pinto-JMAA-1997} where the authors study the so called $(h,k)-$trichotomies and the nonuniform exponential trichotomies have been consider by Barreira and Valls in~\cite{Barreira_Valls-JMPA-2005,Barreira_Valls-ETDS-2006}.

Hence, it is natural to consider trichotomies that are both nonuniform and nonexponential. This was done by Barreira and Valls in~\cite{Barreira_Valls-CPAA-2010,Barreira_Valls-SPJMS-2011} where have been introduced the so-called $\rho-$nonuniform exponential trichotomies, but these trichotomies do not include as a particular case the $(h,k)-$tri\-cho\-to\-mies of Fenner and Pinto \cite{Fenner_Pinto-JMAA-1997}.

In this paper we are going to consider a very general type of trichotomies that includes as particular cases all the types of trichotomies mentioned above, as well new types of trichotomies. We only consider that the linear equation admits an invariant splitting in three invariant subspaces and the norms of the linear evolution operator composed with the three different projections are bounded by general functions that only depend on the initial and on the final time (see~\ref{eq:trich:D1},~\ref{eq:trich:D2} and~\ref{eq:trich:D3}). Despite of that we were able to obtain invariant manifolds provided that the Lipschitz constants of the perturbation are sufficiently small. Note that for dichotomies this has already been done in~\cite{Bento_Silva-BSM-2014} and in~\cite{Bento_Silva-PM-2016} for differential and for difference equations, respectively.

The proof of the main theorem of this paper is based in the so-called classical Lyapunov-Perron method (see \cite{Lyapunov-IJC-1892-1992, Perron-MZ-1929, Perron-MZ-1930}) that consists in the following:
\begin{enumerate}[\lb=$-$,\lm=5mm,\ts=0mm]
   \item the variation of constants formula that allows to relate the solutions of the linear equation with the solutions of the perturbed equation;
   \item the construction of a suitable space of functions that is a complete metric space;
   \item the construction of a suitable contraction on the complete metric space mentioned above;
   \item the application of Banach's fixed point theorem to the mentioned contraction gives a function that is the only fixed point of the contraction and whose graph is the invariant manifold.
\end{enumerate}
This method was used by many authors, namely~\cite{Barreira_Valls-ETDS-2006, Barreira_Valls-CPAA-2010, Bento_Silva-BSM-2014, Bento_Silva-PM-2016}.
However, in this paper we have introduced a novelty in the application of the method. In \cite{Barreira_Valls-ETDS-2006, Barreira_Valls-CPAA-2010, Bento_Silva-BSM-2014, Bento_Silva-PM-2016} are used two applications of the Banach's fixed point theorem, the first one to obtain the solutions of the perturbed equation along the stable/center direction and the other to obtain the solutions of the perturbed equation in the other directions. Here, with only one application of the Banach's fixed point theorem, we obtain the solutions of the perturbed equation in all directions.

As particular cases of our main result we improve the results obtained by Barreira and Valls \cite{Barreira_Valls-CPAA-2010, Barreira_Valls-ETDS-2006} for the so-called $\rho-$nonuniform exponential trichotomies and non\-u\-ni\-form exponential trichotomies, respectively. Moreover, we also obtain as particular cases new results for nonuniform $\nabcd$tricotomies and for $\mu-$non\-u\-ni\-form polynomial trichotomies, concepts that have been introduced for the first time in this paper.

The structure of the paper is as follows. In Section~\ref{sec:notation+preliminaries} we introduce the notation and preliminearies. The main theorem of the paper is stated in Section~\ref{sec:main-result} and in Section~\ref{sec:examples} we apply our main result to particular cases of trichotomies. Finally, in Section~\ref{sec:proof:thm|global}, we prove the main theorem.

\section{Notation and preliminaries}\label{sec:notation+preliminaries}

Let $X$ be a Banach space, let $B(X)$ be the Banach algebra of all bounded linear operators acting on $X$ and let $A \colon \R \to B(X)$ be a continuous map. Consider the linear differential equation
\begin{equation}\label{eq:v'=A(t)v}
   v'=A(t)v, \qquad v(s)=v_{s}
\end{equation}
with $s\in\R$ and $v_{s}\in X$. We are going to assume that~\eqref{eq:v'=A(t)v} has a global solution and denote by $T_{t,s}$ the linear evolution operator associated to equation~\eqref{eq:v'=A(t)v}, i.e.,
\begin{equation*}
   v(t) = T_{t,s}v(s)
\end{equation*}
for every $t$, $s \in \R$.

We say that~\eqref{eq:v'=A(t)v} admits an \textit{invariant splitting} if, for every $t \in \R$, there exist bounded projections $P_{t}$, $Q^+_{t}$, $Q^-_{t} \in B(X)$ such that
\begin{enumerate}[\lb=$\mathbf{(S\arabic*)}$,\lm=14mm]
   \item \label{eq:split:S1}
      $P_t + Q^+_t + Q^-_t = \Id$ for every $t \in \R$;
   \item \label{eq:split:S2}
      $P_t Q^+_t = 0$ for every $t \in \R$;
   \item \label{eq:split:S3}
      $P_t T_{t,s}=T_{t,s}P_{s}$ for every $t,s \in \R$;
   \item \label{eq:split:S4}
      $Q^+_t T_{t,s}=T_{t,s}Q^+_s$  for every $t,s \in \R$.
\end{enumerate}
From~\ref{eq:split:S1} and~\ref{eq:split:S2} we have
   $$ P_t Q^-_t = Q^+_t P_t = Q^-_t P_t = Q^+_t Q^-_t = Q^-_t Q^+_t = 0
      \ \ \text{ for every } t \in \R$$
and from~\ref{eq:split:S1},~\ref{eq:split:S3} and~\ref{eq:split:S4} it follows immediately that
   $$ Q^-_t T_{t,s}=T_{t,s} Q^-_s \ \ \text{ for every } t,s \in \R.$$

For each $t\in \R$, we define the linear subspaces $E_t = P_t (X)$, $F^+_t = Q^+_t (X)$ and $F^-_t=Q^-_t (X)$, and, as usual, we identify $E_t \times F^+_t \times F^-_t$ and $E_t \oplus F^+_t \oplus F^-_t = X$ as the same vector space.

Given functions $\alpha \colon \R^2 \to \R^+$, $\beta^+ \colon \R^2_\ge \to \R^+$ and $\beta^- \colon \R^2_\le \to \R^+$, where
   $$ \R^2_\ge = \set{(t,s) \in \R^2  \colon t\ge s}
      \ \ \ \ \text{ and } \ \ \ \
      \R^2_\le = \set{(t,s) \in \R^2  \colon t\le s},$$
and denoting $\alpha(t,s)$, $\beta^+(t,s)$ and $\beta^-(t,s)$ by $\alpha_{t,s}$, $\beta^+_{t,s}$ and $\beta^-_{t,s}$, respectively, we say that equation~\eqref{eq:v'=A(t)v} admits a \textit{generalized trichotomy} with bounds $\alpha=\prts{\alpha_{t,s}}_{\prts{t,s} \in \R^2}$, $\beta^+ =\prts{\beta^+_{t,s}}_{\prts{t,s} \in \R^2_\ge}$ and $\beta^- = \prts{\beta^-_{t,s}}_{\prts{t,s} \in \R^2_\le}$, or simply with bounds $\alpha_{t,s}$, $\beta^+_{t,s}$ and $\beta^-_{t,s}$, if it admits an invariant splitting such that
\begin{enumerate}[\lb=$\mathbf{(D\arabic*)}$,\lm=14mm]
   \item \label{eq:trich:D1}
      $\norm{T_{t,s}P_s} \le \alpha_{t,s}$ for every $(t,s) \in \R^2$;
   \item \label{eq:trich:D2}
      $\norm{T_{t,s} Q^+_s} \le \beta^+_{t,s}$ for every $(t,s) \in \R^2_\ge$;
   \item \label{eq:trich:D3}
      $\norm{T_{t,s} Q^-_s} \le \beta^-_{t,s}$ for every $(t,s) \in \R^2_\le$.
\end{enumerate}

\begin{example}
   Let
      $$ \ga, \gb, \gc, \gd \colon \R \to \ ]0,+\infty[$$
   be $C^1$ functions and let
      $$ \epsa, \epsb, \epsc, \epsd \colon \R \to [1, +\infty[$$
   be $C^1$ functions in $\R \setm{0}$ and with derivatives from the left and from the right at $t=0$. In $\R^4$, equipped with the maximum norm, consider the differential equation
   \begin{equation}\label{eq:v'=A(t)v:R^4}
      \begin{cases}
         u' = \prtsr{-\dfrac{\ga'(t)}{\ga(t)} + \epsa^*(t)} \, u,\\[4mm]
         v' = \prtsr{\dfrac{\gc'(t)}{\gc(t)}  + \epsc^*(t)} \, v,\\[4mm]
         w' = \prtsr{-\dfrac{\gd'(t)}{\gd(t)} + \epsd^*(t)} \, w,\\[4mm]
         z' = \prtsr{\dfrac{\gb'(t)}{\gb(t)}  + \epsb^*(t)} \, z.
      \end{cases}
   \end{equation}
   where
      $$ \epsi^*(t)
         =
         \begin{cases}
            \dfrac{\epsi'(t)}{\epsi(t)} \, \dfrac{\cos t-1}{2}
                  - \ln(\epsi(t)) \, \dfrac{\sin t}{2}
               & \text{ if } t \ne 0,\\
            0
               & \text{ if } t = 0,
         \end{cases}$$
   for $i=\ga, \gb, \gc, \gd$. The evolution operator of this equation is given by
      $$ T_{t,s}(u,v,w,z)  = \prts{U_{t,s}(u,v), V^+_{t,s}w, V^-_{t,s} z}$$
   where $U_{t,s} \colon \R^2 \to \R^2$ is defined by
      $$ U_{t,s}(u,v)
         = \prts{\dfrac{\ga(s)}{\ga(t)} \, \dfrac{\epsa(t)^{\prts{\cos t-1}/2}}
            {\epsa(s)^{\prts{\cos s-1}/2}} \, u,
         \dfrac{\gc(t)}{\gc(s)} \, \dfrac{\epsc(t)^{\prts{\cos t-1}/2}}
            {\epsc(s)^{\prts{\cos s-1}/2}} \, v}$$
   and $V^+_{t,s}, V^-_{t,s} \colon \R \to \R$ are defined by
      $$ V^+_{t,s} w
         = \dfrac{\gd(s)}{\gd(t)} \, \dfrac{\epsd(t)^{\prts{\cos t-1}/2}}
            {\epsd(s)^{\prts{\cos s-1}/2}} \, w
         \ \ \ \ \text{ and } \ \ \ \
         V^-_{t,s} z
         = \dfrac{\gb(t)}{\gb(s)} \, \dfrac{\epsb(t)^{\prts{\cos t-1}/2}}
            {\epsb(s)^{\prts{\cos s-1}/2}} \, \, z.$$

   Using the projections 
   $P_s(u,v,w,z) = (u,v,0,0)$, $Q^+_s(u,v,w,z)=(0,0,w,0)$ and $Q^-_s(u,v,w,z)=(0,0,0,z)$ we obtain for every $(t,s) \in \R^2$,
      $$ \norm{T_{t,s} Q^+_s}
         = \norm{V^+_{t,s}}
         = \dfrac{\gd(s)}{\gd(t)} \, \dfrac{\epsd(t)^{\prts{\cos t-1}/2}}
            {\epsd(s)^{\prts{\cos s-1}/2}}
         \le \dfrac{\gd(s)}{\gd(t)} \, \epsd(s)$$
   and
      $$ \norm{T_{t,s} Q^-_s}
         = \norm{V^-_{t,s}}
         = \dfrac{\gb(t)}{\gb(s)} \,
            \dfrac{\epsb(t)^{\prts{\cos t-1}/2}}{\epsb(s)^{\prts{\cos s-1}/2}}
         \le \dfrac{\gb(t)}{\gb(s)} \, \epsb(s).$$
   Morover, assuming that
   \begin{equation}\label{t>=s=>(ga(s)gc(s))/(ga(t)gc(t)...>=1}
      \dfrac{\ga(s) \gc(s)}{\ga(t) \gc(t)}
         \pfrac{\epsa(t)}{\epsc(t)}^{(\cos t - 1)/2}
         \pfrac{\epsc(s)}{\epsa(s)}^{(\cos s - 1)/2}
      \ge 1
      \ \text{ for every } (t,s) \in \R^2_\ge,
   \end{equation}
   we have
   \begin{align*}
      \norm{T_{t,s}P_s}
      & =
         \begin{cases}
            \dfrac{\ga(s)}{\ga(t)} \,
               \dfrac{\epsa(t)^{(\cos t - 1)/2}}{\epsa(s)^{(\cos s - 1)/2}}
               & \text{  for all } (t,s) \in \R^2_\ge,\\[3mm]
            \dfrac{\gc(t)}{\gc(s)} \,
               \dfrac{\epsc(t)^{(\cos t - 1)/2}}{\epsc(s)^{(\cos s - 1)/2}}
               & \text{  for all } (t,s) \in \R^2_\le,
         \end{cases}\\
      & \le
         \begin{cases}
            \dfrac{\ga(s)}{\ga(t)} \, \epsa(s)
               & \text{  for all } (t,s) \in \R^2_\ge,\\[3mm]
            \dfrac{\gc(t)}{\gc(s)} \, \epsc(s)
               & \text{  for all } (t,s) \in \R^2_\le.
         \end{cases}
   \end{align*}
   Therefore, if~\eqref{t>=s=>(ga(s)gc(s))/(ga(t)gc(t)...>=1} is satisfied, then equation~\eqref{eq:v'=A(t)v:R^4} has a generalized trichotomy  with bounds
   \begin{equation} \label{eq:alpha_t,s=a(s)/a(t).epsa(s)...}
      \begin{aligned}
            \alpha_{t,s}
         &
            =
            \begin{cases}
               \dfrac{\ga(s)}{\ga(t)}\,\epsa(s)\\[3mm]
               \mins{\epsa(s),\epsc(s)}\\[2mm]
               \dfrac{\gc(t)}{\gc(s)}\,\epsc(s)
            \end{cases}
            \!\!\!\!\!\!
         &&
            \!\!
            \begin{tabular}{l}
               for all $(t,s) \in \R^2_\ge$ with $t \ne s$,\\[4.5mm]
               for all $(t,s) \in \R^2$ with $t = s$,\\[4.5mm]
               for all $(t,s) \in \R^2_\le$ with $t \ne s$,
            \end{tabular}
         \\
            \beta^+_{t,s}
         &
            = \dfrac{\gd(s)}{\gd(t)}\,\epsd(s)
         &&
            \text{for all } (t,s) \in \R^2_\ge,
         \\
            \beta^-_{t,s}
         &
            = \dfrac{\gb(t)}{\gb(s)}\,\epsb(s)
         &&
            \text{for all } (t,s) \in \R^2_\le.
      \end{aligned}
   \end{equation}
   We call the trichotomies with bounds of this type \textit{nonuniform $\prts{\ga,\gb,\gc,\gd}-$tri\-cho\-to\-mies}.
\end{example}

\begin{example}
   Let $\rho \colon \R \to \R$ an odd increasing differentiable function such that
      $$ \dlim_{t \to +\infty} \rho(t) = +\infty.$$
   In~\eqref{eq:alpha_t,s=a(s)/a(t).epsa(s)...}, making
      $$ \ga(t) = \e^{-a \rho(t)}, \ \ \
         \gc(t) = \e^{-c \rho(t)}, \ \ \
         \gb(t) = \e^{-b \rho(t)} , \ \ \
         \gd(t) = \e^{-d \rho(t)}$$
   and
      $$ \epsa(t)
         = \epsb(t)
         =\epsc(t)
         = \epsd(t)
         = D \e^{\eps \abs{\rho(t)}},$$
   with $a, b, c, d, D, \eps \in \R$ such that $D \ge 1$ and $\eps \ge 0$, we get
   \begin{equation} \label{eq:alpha_t,s=De^[a(rho(t)-rho(s))+eps|rho(s)|]}
      \begin{split}
         \alpha_{t,s}
         & =
            \begin{cases}
               D \e^{a[\rho (t)-\rho (s)] + \eps|\rho(s)|}\\
               D \e^{c[\rho (s)-\rho (t)] + \eps|\rho(s)|}\\
            \end{cases}\\
         \beta^+_{t,s}
            & = D \e^{d [\rho(t) - \rho(s)] + \eps|\rho(s)|}\\
         \beta^-_{t,s}
            & = D \e^{b [\rho(s) - \rho(t)] + \eps|\rho(s)|}\\
      \end{split}\!\!\!\!
      \begin{split}
         \text{ for all } (t,s) \in \R^2_\ge,\\
         \text{ for all } (t,s) \in \R^2_\le,\\
         \text{ for all } (t,s) \in \R^2_\ge,\\
         \text{ for all } (t,s) \in \R^2_\le.
      \end{split}
   \end{equation}
   This bounds for the trichotomy were consider by Barreira and Valls in~\cite{Barreira_Valls-SPJMS-2011, Barreira_Valls-CPAA-2010} and are called \textit{$\rho-$nonuniform exponential trichotomies}. Note that in this case condition~\eqref{t>=s=>(ga(s)gc(s))/(ga(t)gc(t)...>=1} is equivalent to $ a+c \ge 0$.

   When $\rho(t) = t$ we obtain the \textit{nonuniform exponential trichotomies} consider by Barreira and Valls in~\cite{Barreira_Valls-ETDS-2006} and~\cite{Barreira_Valls-JMPA-2005} with the bounds of the form
   \begin{equation} \label{eq:alpha_t,s = D e^[a(t-s)+eps|s|]}
      \begin{split}
         \alpha_{t,s}
         & =
            \begin{cases}
               D \e^{a(t-s) + \eps|s|}\\
               D \e^{c(s-t) + \eps|s|}\\
            \end{cases}\\
         \beta^+_{t,s}
            & = D \e^{d(t-s) + \eps|s|}\\
         \beta^-_{t,s}
            & = D \e^{b(s-t) + \eps|s|}\\
      \end{split}\!\!\!\!
      \begin{split}
         \text{ for all } (t,s) \in \R^2_\ge,\\
         \text{ for all } (t,s) \in \R^2_\le,\\
         \text{ for all } (t,s) \in \R^2_\ge,\\
         \text{ for all } (t,s) \in \R^2_\le.
      \end{split}
   \end{equation}
\end{example}
\section{Main theorem}\label{sec:main-result}
Suppose that equation~\eqref{eq:v'=A(t)v}  admits a generalized trichotomy. Consider the initial value problem
\begin{equation}\label{eq:v'=A(t)v+f(t,v)}
   v'=A(t)v+f(t,v), \qquad v(s)=v_{s}
\end{equation}
where $f: \R\times X\rightarrow X $ is a continuous function such that
\begin{equation}\label{eq:f(t,0)=0}
   f(t,0)=0 \text{ for every } t \in \R
\end{equation}
and, for every $t \in \R$,
\begin{equation}\label{eq:f is Lip}
   \Lipft
   := \sups{\dfrac{\norm{f(t,x)-f(t,y)}}{\norm{x-y}}
      \colon x,y \in X, \ x \ne y}
   < +\infty,
\end{equation}
i.e., the function $f_t \colon X \to X$ given by $f_{t}(x)=f(t,x)$ is a Lipschitz function (in $x$). Clearly
\begin{equation}\label{eq:||f(t,x)-f(t,y)||<=...}
   \norm{f(t,x)-f(t,y)} \le \Lipft \norm{x-y}
\end{equation}
for every $x,y \in X$ and every $t \in \R$ and taking $y=0$ in the last inequality, and by~\eqref{eq:f(t,0)=0}, we have
\begin{equation}\label{eq:||f(t,x)||<=...}
   \norm{f(t,x)} \le \Lipft \norm{x}
\end{equation}
for every $x \in X$ and every $t \in \R$.

When~\eqref{eq:v'=A(t)v} admits a generalized trichotomy, we can write the only solution of~\eqref{eq:v'=A(t)v+f(t,v)} in the form
\begin{equation*}
   \prts{x(t,s,v_s),y^+(t,s,v_s),y^-(t,s,v_s)} \in E_t \times F^+_t  \times F^-_t,
\end{equation*}
where $v_{s}=(\xi,\eta^+,\eta^-)\in E_s \times F^+_s \times F^-_s$, and then solving problem~\eqref{eq:v'=A(t)v+f(t,v)} is equivalent to solve the following problem
\begin{align}
   & x(t)
      = T_{t,s}P_s \xi + \int^t_s T_{t,r}P_r f(r,x(r),y^+(r),y^-(r)) \dr
      \label{eq:x(t)=T_t,sP_s...}\\
   & y^+(t)
      = T_{t,s}Q^+_s \eta^+ + \int^t_s T_{t,r}Q^+_r f(r,x(r),y^+(r),y^-(r))\dr
      \label{eq:y^+(t)=T_t,sQ^+_s...}\\
   & y^-(t)
      = T_{t,s}Q^-_s \eta^- + \int^t_s T_{t,r}Q^-_r f(r,x(r),y^+(r),y^-(r))\dr
      \label{eq:y^-(t)=T_t,sQ^-_s...}
\end{align}
for every $t \in \R$.

For each $\tau \in \R$, we define the \textit{flow} by
\begin{equation}\label{eq:Psi(s,v_s)=...}
   \Psi_\tau (s,v_{s})
   = \prts{s+\tau ,x(s+\tau,s,v_{s}),y^+(s+\tau,s,v_{s}),y^-(s+\tau,s,v_{s})}.
\end{equation}

We are going to study the existence of invariant center manifolds for equation~\eqref{eq:v'=A(t)v+f(t,v)} when~\eqref{eq:v'=A(t)v} admits a generalized trichotomy. The invariant center manifolds that we are going to obtain are given as the graph of a function belonging to a certain function space that we define now.

Making
\begin{equation*}
   G = \bigcup_{t \in \R}\left\{t\right\}\times E_{t},
\end{equation*}
and choosing $N \in \ ]0,+\infty[$, we denote by $\cX_N $ the space of all continuous functions $\phi \colon G \to X $ such that
\begin{align}
   & \phi(t,0)= 0 \ \text{ for all } t \in \R;
      \label{eq:phi(t,0)=0}\\
   & \phi(t,\xi) \in  F^+_t \oplus F^-_t
      \ \text{ for all } (t,\xi) \in G;
      \label{eq:phi(t,xi)_in_F^+_t+F^-_t}\\
   & \sups{\dfrac{\norm{\phi(t,\xi) -\phi(t,\bxi)}}{\norm{\xi-\bxi}} \colon (t,\xi), (t,\bxi) \in G, \ \xi\ne\bxi} \le N.
      \label{eq:Lip(phi_t)<=N}
\end{align}

Note that from~\eqref{eq:Lip(phi_t)<=N} it follows immediately that
\begin{equation}\label{eq:||phi(t,xi)-phi(t,bxi)||<=...}
   \norm{\phi(t,\xi) -\phi(t,\bxi)}
   \le N \norm{\xi-\bxi} \text{ for all } (t,\xi), (t,\bxi) \in G
\end{equation}
 and making $\bxi=0$ in~\eqref{eq:||phi(t,xi)-phi(t,bxi)||<=...}, we have
\begin{equation}\label{eq:||phi(t,xi)||<=...}
   \norm{\phi(t,\xi)} \le N \norm{\xi} \text{ for every $(t,\xi) \in G$.}
\end{equation}

By~\eqref{eq:phi(t,xi)_in_F^+_t+F^-_t}, and identifying $F^+_t \oplus F^-_t$ and $F^+_t \times F^-_t$ as the same vector space,  we can write $\phi=\prts{\phi^+,\phi^-}$, where $\phi^+(t,\xi) = Q^+_t \phi(t,\xi) $ and $\phi^-(t,\xi) = Q^-_t \phi(t,\xi) $.

Given $\phi \in \cX_N$, we define the \textit{graph of $\phi$} as
\begin{equation}\label{def:V_phi}
   \begin{split}
      \cV_{\phi}
      & = \set{\prts{s, \xi, \phi(s,\xi)} \colon (s,\xi) \in G}\\
      & = \set{\prts{s, \xi, \phi^+(s,\xi) , \phi^-(s,\xi)}
         \colon (s,\xi) \in G}\\
      & \subseteq \R \times X.
   \end{split}
\end{equation}
The global Lipschitz invariant center manifolds that we intend to obtain are given as the graph of suitable functions $\phi$ belonging to some space $\cX_N$.

Before state the main theorem we need to define the following quantities:
\begin{equation}\label{def:sgm}
   \sgm
   := \sup_{(t,s) \in \R^2} \abs{\dint_s^t
      \dfrac{\alpha_{t,r} \Lipfr \alpha_{r,s}}{\alpha_{t,s}} \dr}
\end{equation}
and
\begin{equation}\label{def:omg}
   \omg
   := \sup_{s\in\R}
      \prtsr{\dint_{-\infty}^s \beta^+_{s,r} \Lipfr \alpha_{r,s} \dr
      + \dint_s^{+\infty}  \beta^-_{s,r} \Lipfr \alpha_{r,s} \dr}
\end{equation}
that are supposed to be finite.

\begin{theorem} \label{thm:global}
   Let $X$ be a Banach space. Suppose that~\eqref{eq:v'=A(t)v} admits a generalized trichotomy with bounds $\alpha_{t,s}$, $\beta^+_{t,s}$ and $\beta^-_{t,s}$ and let $ f \colon \R \times X \to X$ be a continuous function such that~\eqref{eq:f(t,0)=0} and~\eqref{eq:f is Lip} are satisfied. If
   \begin{equation} \label{eq:lim+oo b^- a = lim-oo b^+ a =0}
      \lim_{r \to +\infty} \beta^-_{s,r} \alpha_{r,s}
      = \lim_{r \to -\infty} \beta^+_{s,r} \alpha_{r,s}
      =0
      \ \ \text{ for every } s \in \R
   \end{equation}
   and
   \begin{equation*}
      2 \sgm + 2 \omg < 1,
   \end{equation*}
   where $\sgm$ and $\omg$ are given by~\eqref{def:sgm} and~\eqref{def:omg}, respectively, then there is $N \in \ ]0,1[$ and a unique $\phi \in \cX_N$ such that
   \begin{equation}\label{eq:psi_tau(V)_C=_V}
      \Psi_\tau(\cV_\phi)\subset \cV_\phi
   \end{equation}
   for every $\tau \in \R$, where $\Psi_\tau$ is given by~\eqref{eq:Psi(s,v_s)=...} and $\cV_\phi$ is given by~\eqref{def:V_phi}.

   Moreover,
      $$ \norm{\Psi_{t-s}(s,\xi,\phi(s,\xi))
            - \Psi_{t-s}(s,\bxi,\phi(s,\bxi))}
         \le \dfrac{N}{\omg} \, \alpha_{t,s} \, \norm{\xi-\bxi}$$
   for all $t \in \R$ and all $(s,\xi), (s,\bxi) \in G$.
\end{theorem}

The proof of last theorem will be given in Section~\ref{sec:proof:thm|global}.

\section{Particular cases of the main theorem}\label{sec:examples}

Now we will apply the main result to nonuniform $\nabcd$trichtomies.

\begin{theorem} \label{thm:global:alpha_t,s=ga(s)/ga(t)eps_a(s)...}
   Let $X$ be a Banach space. Suppose that equation~\eqref{eq:v'=A(t)v} admits nonuniform $\nabcd$trichotomy and let $f \colon \R \times X \to X$ be a continuous function such that~\eqref{eq:f(t,0)=0} and~\eqref{eq:f is Lip} are satisfied and
   \begin{equation}\label{ine:Lip(fr)<delta.eps.min(...)}
      \Lip(f_r)
      \le \delta
         \mins{\dfrac{1}{\gc(r)\gd(r)\epsd(r)}\, \prfrac{\gc(r)\gd(r)}{\epsc(r)}',
         \dfrac{\ga(r)\gb(r)}{\epsb(r)} \,
         \prtsr{-\dfrac{1}{\ga(r)\gb(r)\epsa(r)}}', \gamma(r)}
   \end{equation}
   for every $r \in \R \setm{0}$, where $\delta < 1/6$ and $\gamma \colon \R \to \, ]0,+\infty[$ is a function such that
   \begin{equation}\label{ine:int gamma dr < 1}
      \maxs{\int_{-\infty}^{+\infty} \epsa(r) \gamma(r) \dr,
         \int_{-\infty}^{+\infty} \epsc(r) \gamma(r) \dr}
      \le 1.
   \end{equation}
   If
   \begin{equation}\label{eq:lim_c(r)d(r)eps(r) =...=0}
      \lim_{r \to -\infty} \gc(r) \gd(r) \epsd(r)
      = \lim_{r \to + \infty}  \dfrac{\epsb(r)}{\ga(r)\gb(r)} = 0,
   \end{equation}
   then equation~\eqref{eq:v'=A(t)v+f(t,v)} admits a global Lipschitz invariant center manifold, i.e., there is $N \in \ ]0,1[$ and a unique $\phi \in \cX_N$ such that
   \begin{equation*}
      \Psi_\tau(\cV_\phi)\subset \cV_\phi \ \ \text{ for every } \tau \in \R
   \end{equation*}
    and
   \begin{equation*}
      \norm{\Psi_{t-s}(p_{s,\xi}) - \Psi_{t-s}(p_{s,\bxi})}
      \le
      \begin{cases}
         \dfrac{N}{\omg} \, \dfrac{\ga(s)}{\ga(t)} \, \epsa(s) \, \norm{\xi-\bxi}
            \ \ \text{ if } t \ge s,\\[3mm]
         \dfrac{N}{\omg} \, \dfrac{\gc(t)}{\gc(s)} \, \epsc(s) \, \norm{\xi-\bxi}
            \ \ \text{ if } t \le s,
      \end{cases}
   \end{equation*}
   for all $(s,\xi), (s,\bxi) \in G$ and where $p_{s,\xi} = \prts{s,\xi,\phi(s,\xi)}$ and $p_{s,\bxi} = \prts{s,\bxi,\phi(s,\bxi)}$.
\end{theorem}

\begin{proof}
   It is obvious that for this type of bounds~\eqref{eq:lim+oo b^- a = lim-oo b^+ a =0} is equivalent to~\eqref{eq:lim_c(r)d(r)eps(r) =...=0}. Moreover, from~\eqref{ine:Lip(fr)<delta.eps.min(...)} and~\eqref{ine:int gamma dr < 1} we have
   \begin{align*}
      \sgm
      & = \sup_{(t,s) \in \R^2}
         \abs{\dint_s^t \dfrac{\alpha_{t,r} \Lip (f_r) \alpha_{r,s}}{\alpha_{t,s}} \dr}\\
      & = \maxs{\sup_{(t,s) \in \R^2_\ge} \int_s^t \epsa(r) \Lip (f_r) \dr,
         \sup_{(t,s) \in \R^2_\le} \int_t^s \epsc(r) \Lip (f_r) \dr}\\
      & = \maxs{\int_{-\infty}^{+\infty} \epsa(r) \Lip (f_r) \dr,
         \int_{-\infty}^{+\infty} \epsc(r) \Lip (f_r) \dr}\\
      & \le \delta \maxs{\int_{-\infty}^{+\infty} \epsa(r) \gamma(r) \dr,
         \int_{-\infty}^{+\infty} \epsc(r) \gamma(r) \dr}\\
      & \le \delta.
   \end{align*}
   From~\eqref{ine:Lip(fr)<delta.eps.min(...)} and~\eqref{eq:lim_c(r)d(r)eps(r) =...=0} it follows that
   \begin{align*}
      \omg
      & = \sup_{s\in\R} \prtsr{
         \dint_{-\infty}^s \beta^+_{s,r} \Lipfr \alpha_{r,s} \dr
         + \dint_s^{+\infty}  \beta^-_{s,r} \Lipfr \alpha_{r,s} \dr}\\
      & = \sup_{s \in \R} \prtsr{\dint_{-\infty}^s
         \dfrac{\gd(r) \epsd(r) \gc(r) \epsc(s)}{\gd(s)\gc(s)} \Lip(f_r) \dr + \dint_s^{+\infty} \dfrac{\gb(s) \epsb(r) \ga(s) \epsa(s)}{\gb(r)\ga(r)} \Lip(f_r) \dr}\\
      & \le \delta\prtsr{
         \dfrac{\epsc(s)}{\gc(s)\gd(s)}
         \int_{-\infty}^s \prtsr{\dfrac{\gc(r)\gd(r)}{\epsc(r)}}' \dr
         - \ga(s)\gb(s)\epsa(s)
         \int_s^{+\infty} \prtsr{\dfrac{1}{\ga(r)\gb(r)\epsa(r)}}' \dr}\\
      & = 2 \delta.
   \end{align*}
   Hence, since $\delta < 1/6$ we have $2 \sigma + 2 \omega \le 6 \delta < 1$ and all hypothesis of Theorem~\ref{thm:global} are satisfied.
\end{proof}

\begin{remarks}\
   \begin{enumerate}[\lb=$\alph*)$,\lm=5mm]
      \item From~\eqref{ine:Lip(fr)<delta.eps.min(...)} it follows that for every $r \ne 0$ we must have
         \begin{equation*}
            \prtsr{\dfrac{\gc(r) \gd(r)}{\epsc(r)}}'
            = \dfrac{\gc(r)\gd(r)}{\epsc(r)}
               \prtsr{\dfrac{\gc'(r)}{\gc(r)} + \dfrac{\gd'(r)}{\gd(r)}
               - \dfrac{\epsc'(r)}{\epsc(r)}}
            > 0
         \end{equation*}
         and
         \begin{equation*}
            \prtsr{-\dfrac{1}{\ga(r)\gb(r)\epsa(r)}}'
            = \dfrac{1}{\ga(r)\gb(r)\epsa(r)}
               \prtsr{\dfrac{\ga'(r)}{\ga(r)} + \dfrac{\gb'(r)}{\gb(r)}
               + \dfrac{\epsa'(r)}{\epsa(r)}}
            > 0.
         \end{equation*}
      \item Note that the function $\gamma$ in last theorem can be chosen as
            $$ \gamma(t)
               = \dfrac{\mins{\dfrac{\epsc(s)}{2\gc(s)\gd(s)}
                     \prfrac{\gc(t)\gd(t)}{\epsc(t)}',
                     \dfrac{\ga(s)\gb(s)\epsa(s)}{2} \prtsr{-\dfrac{1}{\ga(t)\gb(t)\epsa(t)}}'}}
                  {\maxs{\epsa(t),\epsc(t)}}$$
         for $t \ne 0$ and where $s$ is a fixed real number. In fact, from~\eqref{eq:lim_c(r)d(r)eps(r) =...=0} we have
         \begin{align*}
            & \int_{-\infty}^{+\infty} \epsa(r) \gamma(r) \dr\\
            & = \int_{-\infty}^s \epsa(r) \gamma(r) \dr
                  +\int_s^{+\infty} \epsa(r) \gamma(r) \dr\\
            & \le \int_{-\infty}^{s} \dfrac{\epsc(s)}{2\gc(s)\gd(s)}
                     \prtsr{\dfrac{\gc(r)\gd(r)}{\epsc(r)}}' \dr
                  + \int_s^{+\infty} \dfrac{\ga(s)\gb(s)\epsa(s)}{2}
                     \prtsr{\dfrac{-1}{\ga(r)\gb(r)\epsa(r)}}'\dr\\
            & = 1
         \end{align*}
         and
         \begin{align*}
            & \int_{-\infty}^{+\infty} \epsc(r) \gamma(r) \dr\\
            & = \int_{-\infty}^s \epsc(r) \gamma(r) \dr
                  +\int_s^{+\infty} \epsc(r) \gamma(r) \dr\\
            & \le \dint_{-\infty}^s \dfrac{\epsc(s)}{2\gc(s)\gd(s)}
                     \prtsr{\dfrac{\gc(r)\gd(r)}{\epsc(r)}}' \dr
                  + \dint_s^{+\infty} \dfrac{\ga(s)\gb(s)\epsc(s)}{2}
                     \prtsr{\dfrac{-1}{\ga(r)\gb(r)\epsa(r)}}'\dr\\
            & = 1.
         \end{align*}
   \end{enumerate}
\end{remarks}

\begin{theorem}\label{thm:rho-ned}
   Let $X$ be a Banach space and suppose that equation~\eqref{eq:v'=A(t)v} admits a trichotomy with bounds of the form~\eqref{eq:alpha_t,s=De^[a(rho(t)-rho(s))+eps|rho(s)|]}. Suppose that $f \colon \R \times X \to X$ is a continuous function that satisfies~\eqref{eq:f(t,0)=0} and~\eqref{eq:f is Lip} and
    \begin{equation}\label{eq:Lip(f_r)<= for alpha_t,s with rho}
      \Lipfr
      \le \dfrac{\delta \rho'(r)}{D^2 \e^{2 \eps\abs{\rho(r)}}}
          \mins{-c-d-\eps \sgn(r), -a-b+\eps \sgn(r), \dfrac{D \gamma}{2}\e^{(\eps-\gamma)|\rho(r)|}},
   \end{equation}
   where $\delta < 1/6$ and
   \begin{equation}\label{eq:gamma=eps or gamma>0}
      \gamma = \eps \ \text{ if } \ \eps > 0
      \ \ \ \ \text{ and } \ \ \ \
      0 < \gamma < \dfrac{2}{D} \mins{-c-d,-a-b} \ \text{ if } \ \eps = 0.
   \end{equation}
   If
   \begin{equation}\label{ine:a+b+eps<0|c+d+eps<0}
      a + b +\eps < 0 \ \ \ \ \text{ and } \ \ \ \
      c+d+\eps < 0,
   \end{equation}
   then~\eqref{eq:v'=A(t)v+f(t,v)} admits a global Lipschtiz invariant center manifold, i.e., there is $N \in \ ]0,1[$ and a unique $\phi \in \cX_N$ such that
   \begin{equation*}
      \Psi_\tau(\cV_\phi)\subset \cV_\phi \ \ \text{ for every } \tau \in \R
   \end{equation*}
    and
   \begin{equation}\label{ine:||Psi.(.,xi)-Psi.(.,bxi)||<=...|rho-net}
      \norm{\Psi_{t-s}(p_{s,\xi}) - \Psi_{t-s}(p_{s,\bxi})}
         \le
         \begin{cases}
            \dfrac{DN}{\omg} \, \e^{a\prtsr{\rho(t)-\rho(s)}+\eps\abs{\rho(s)}} \, \norm{\xi-\bxi}
               \ \ \text{ if } t \ge s,\\[3mm]
            \dfrac{DN}{\omg} \, \e^{c\prtsr{\rho(s)-\rho(t)}+\eps\abs{\rho(s)}} \, \norm{\xi-\bxi}
               \ \ \text{ if } t \le s,
         \end{cases}
   \end{equation}
   for all $(s,\xi), (s,\bxi) \in G$ and where $p_{s,\xi} = \prts{s,\xi,\phi(s,\xi)}$ and $p_{s,\bxi} = \prts{s,\bxi,\phi(s,\bxi)}$.
\end{theorem}

\begin{proof}
   We are going to apply Theorem~\ref{thm:global:alpha_t,s=ga(s)/ga(t)eps_a(s)...}. For this bounds condition~\eqref{eq:lim_c(r)d(r)eps(r) =...=0} is equivalent to~\eqref{ine:a+b+eps<0|c+d+eps<0} because
   \begin{equation*}
      \gc(r) \gd(r) \epsd(r)
      = D\e^{(-c-d)\rho(r)+\eps |\rho(r)|}
      \ \ \ \ \text{ and } \ \ \ \
      \dfrac{\epsb(r)}{\ga(r)\gb(r)}
      = D \e^{(a+b)\rho(r)+\eps |\rho(r)|}.
   \end{equation*}
   Moreover, for $r \ne 0$ we have
   \begin{equation*}
      \begin{split}
         \dfrac{1}{\gc(r)\gd(r)\epsd(r)}\, \prfrac{\gc(r)\gd(r)}{\epsc(r)}'
         & = \dfrac{1}{\epsc(r) \epsd(r)}
            \prts{\dfrac{\gc'(r)}{\gc(r)}
            + \dfrac{\gd'(r)}{\gd(r)}
            - \dfrac{\epsc'(r)}{\epsc(r)}}\\[1mm]
         & = \dfrac{\rho'(r)}{D^2 \e^{2\eps \abs{\rho(r)}}} \,
            \prts{-c-d-\eps\sgn(r)}
      \end{split}
   \end{equation*}
   and
   \begin{equation*}
      \begin{split}
         \dfrac{\ga(r)\gb(r)}{\epsb(r)} \, \prtsr{-\dfrac{1}{\ga(r)\gb(r)\epsa(r)}}'
         & = \dfrac{1}{\epsa(r)\epsb(r)}
            \prts{\dfrac{\ga'(r)}{\ga(r)}
               + \dfrac{\gb'(r)}{\gb(r)}
               + \dfrac{\epsa'(r)}{\epsa(r)}}\\[1mm]
         & = \dfrac{\rho'(r)}{D^2 \e^{2\eps \abs{\rho(r)}}} \,
            \prts{-a-b+\eps\sgn(r)}.
      \end{split}
   \end{equation*}
   Making
   \begin{equation*}
      \gamma(r)
      = \dfrac{\gamma}{2D} \rho'(r) \e^{-(\eps+\gamma)|\rho(r)|},
   \end{equation*}
   with $\gamma$ given by~\eqref{eq:gamma=eps or gamma>0}, and observing that
      $$ \int_{-\infty}^{+\infty} D \e^{\eps\abs{\rho(r)}} \gamma(r) \dr
         = \int_{-\infty}^{+\infty}
            \dfrac{\gamma}{2} \rho'(r) \e^{-\gamma |\rho(r)|}\dr
         = 2 \int_0^{+\infty}
            \dfrac{\gamma}{2} \rho'(r) \e^{-\gamma \rho(r)}\dr
         = 1,$$
   condition~\eqref{ine:Lip(fr)<delta.eps.min(...)} becomes~\eqref{eq:Lip(f_r)<= for alpha_t,s with rho}.
\end{proof}

Note that Theorem~\ref{thm:rho-ned} improves the result of Barreira and Valls~\cite{Barreira_Valls-CPAA-2010} because our result has a better asymptotic behavior. In fact, with our notation, in~\eqref{ine:||Psi.(.,xi)-Psi.(.,bxi)||<=...|rho-net} where we have $a$ and $c$, Barreira and Valls~\cite{Barreira_Valls-CPAA-2010} have $a+2\delta D$ and $c+2\delta D$, respectively.

Making $\rho(t) = t$ in last theorem we have the following result.

\begin{corollary}
   Let $X$ be a Banach space and suppose that equation~\eqref{eq:v'=A(t)v} admits a trichotomy with bounds of the form~\eqref{eq:alpha_t,s = D e^[a(t-s)+eps|s|]}. Suppose that $f \colon \R \times X \to X$ is a continuous function such that~\eqref{eq:f(t,0)=0} and~\eqref{eq:f is Lip} are satisfied and
   \begin{equation*}
      \Lipfr
      \le \dfrac{\delta}{D^2 \e^{2 \eps\abs{r}}}
          \mins{-c-d-\eps \sgn(r), -a-b+\eps \sgn(r), \dfrac{D \gamma}{2}\e^{(\eps-\gamma)|r|}}.
   \end{equation*}
   where $\gamma$ is given by~\eqref{eq:gamma=eps or gamma>0}.
   If~\eqref{ine:a+b+eps<0|c+d+eps<0} is satisfied,
   then~\eqref{eq:v'=A(t)v+f(t,v)} admits a global Lipschtiz invariant center manifold, i.e., there is $N \in \ ]0,1[$ and a unique $\phi \in \cX_N$ such that
   \begin{equation*}
      \Psi_\tau(\cV_\phi)\subset \cV_\phi \ \ \text{ for every } \tau \in \R
   \end{equation*}
    and
   \begin{equation*}
      \norm{\Psi_{t-s}(p_{s,\xi}) - \Psi_{t-s}(p_{s,\bxi})}
         \le
         \begin{cases}
            \dfrac{DN}{\omg} \, \e^{a\prts{t-s}+\eps\abs{s}} \, \norm{\xi-\bxi}
               \ \ \text{ if } t \ge s,\\[3mm]
            \dfrac{DN}{\omg} \, \e^{c\prts{s-t}+\eps\abs{s}} \, \norm{\xi-\bxi}
               \ \ \text{ if } t \le s,
         \end{cases}
   \end{equation*}
   for all $(s,\xi), (s,\bxi) \in G$ and where $p_{s,\xi} = \prts{s,\xi,\phi(s,\xi)}$ and $p_{s,\bxi} = \prts{s,\bxi,\phi(s,\bxi)}$.
\end{corollary}

Again, as in last theorem, we improve the asymptotic behavior of the result obtained by Barreira and Valls in~\cite{Barreira_Valls-ETDS-2006}.

Now we are going to assume that equation~\eqref{eq:v'=A(t)v} admits a generalized trichotomy with bounds of the form
\begin{equation*}
   \begin{split}
      \alpha_{t,s}
      & =
         \begin{cases}
            D (\mu(t)-\mu(s)+1)^a (|\mu(s)|+1)^\eps\\
            D (\mu(s)-\mu(t)+1)^c (|\mu(s)|+1)^\eps\\
         \end{cases}\\
      \beta^+_{t,s}
         & = D (\mu(t)-\mu(s)+1)^d (|\mu(s)|+1)^\eps\\
      \beta^-_{t,s}
         & = D (\mu(s)-\mu(t)+1)^b (|\mu(s)|+1)^\eps\\
   \end{split}
   \begin{split}
      \text{ for all } (t,s) \in \R^2_\ge,\\
      \text{ for all } (t,s) \in \R^2_\le,\\
      \text{ for all } (t,s) \in \R^2_\ge,\\
      \text{ for all } (t,s) \in \R^2_\le,
   \end{split}
\end{equation*}
where $\mu \colon \R \to \R$ is a differentiable function with positive derivative and
   $$ \dlim_{t \to + \infty} \mu(t)
      = \lim_{t \to -\infty} -\mu(t)
      = + \infty,$$
$a,b,c,d, D, \eps$ are real constants with $D \ge 1$ and $\eps \ge 0$. We call this type of trichotomy a \textit{$\mu-$nonuniform polynomial trichotomy}.

\begin{theorem}
   \label{thm:global:alpha_t,s = (mu(t)-mu(s)+1)^a(|mu(s)|+a)^eps...}
   Let $X$ be a Banach space. Suppose that equation~\eqref{eq:v'=A(t)v} admits a $\mu-$nonuniform polynomial trichotomy and let $f \colon \R \times X \to X$ be a continuous function such that~\eqref{eq:f(t,0)=0} and~\eqref{eq:f is Lip} are satisfied and
      $$ \Lip (f_r)\le \delta \mu'(r)(|\mu(r)|+1)^{-\gamma}$$
   with $\gamma > 0$. If~\eqref{ine:a+b+eps<0|c+d+eps<0} is satisfied, $\delta$ is sufficiently small and
      \begin{align*}
         & a, c \le 0, \ \ \
         2\eps \le \gamma \ \ \ \text{ and } \ \ \
         \eps - \gamma + 1 <0,
      \end{align*}
   then~\eqref{eq:v'=A(t)v+f(t,v)} admits a global Lipschitz invariant center manifold, i.e., there is $N \in \ ]0,1[$ and a unique $\phi \in \cX_N$ such that
   \begin{equation*}
      \Psi_\tau(\cV_\phi)\subset \cV_\phi \ \ \text{ for every } \tau \in \R
   \end{equation*}
    and
   \begin{align*}
      & \norm{\Psi_{t-s}(p_{s,\xi}) - \Psi_{t-s}(p_{s,\bxi})}\\
      & \le
         \begin{cases}
            \dfrac{DN}{\omg} \, \prts{\mu(t)-\mu(s)+1}^a \prts{\abs{\mu(s)}+1}^\eps \, \norm{\xi-\bxi}
            & \text{ if } t \ge s,\\[3mm]
            \dfrac{DN}{\omg} \, \prts{\mu(s)-\mu(t)+1}^c \prts{\abs{\mu(s)}+1}^\eps \, \norm{\xi-\bxi}
             & \text{ if } t \le s,
         \end{cases}
   \end{align*}
   for all $(s,\xi), (s,\bxi) \in G$ and where $p_{s,\xi} = \prts{s,\xi,\phi(s,\xi)}$ and $p_{s,\bxi} = \prts{s,\bxi,\phi(s,\bxi)}$.
\end{theorem}

To prove this theorem we need the following lemma.

\begin{lemma}\label{lemma:int_0^+inf (1+t)^lbd (|t+p|+1)^nu...<=...}
   Let $\lbd$, $\nu <0$, $\eps \ge 0$ and $p \in \R$. If
      $$ \lbd+\eps+\nu+1 < 0, \ \ \
         \lbd + \eps \le 0
         \ \ \ \text{ and } \ \ \
         \nu+\eps \le 0,$$
   then
      $$ \dint_0^{+\infty}
            \prts{1+\tau}^\lbd
            \prts{\abs{\tau+p}+1}^\nu
            \prts{\abs{p}+1}^\eps \dtau
         \le
         \begin{cases}
            \dfrac{1}{\abs{\lbd+\eps+\nu+1}}
               & \text{ if \ \ } p \ge 0,\\[4mm]
            \dfrac{2^{\eps+1}+1}{\abs{\lbd+\eps+\nu+1}}
               & \text{ if \ \ } p < 0.
         \end{cases}$$
\end{lemma}

\begin{proof}
   If $p \ge 0$, since $\eps \ge 0$, $\nu+\eps \le 0$ and $\lbd+\eps+\nu+1 < 0$, we have
   \begin{align*}
      \dint_0^{+\infty}
         \prts{\tau+1}^\lbd
         \prts{\abs{\tau+p}+1}^\nu
         \prts{\absp+1}^\eps \dtau
      & = \dint_0^{+\infty}
         \prts{\tau+1}^\lbd
         \prts{\tau+p+1}^\nu
         \prts{p+1}^\eps \dtau\\
      & \le \dint_0^{+\infty}
         \prts{\tau+1}^\lbd
         \prts{\tau+p+1}^{\nu+\eps} \dtau\\
      & \le \dint_0^{+\infty}
         \prts{\tau+1}^{\lbd+\nu+\eps} \dtau\\
      & = \dfrac{1}{\abs{\lbd+\nu+\eps+1}}.
   \end{align*}

   If $p < 0$, then
      $$ \abs{\tau+p}+1
         =
         \begin{cases}
            \abs{p} - \tau + 1 & \text{ if \ \ } 0 \le \tau \le \abs{p},\\
            \tau - \abs{p} + 1 & \text{ if \ \ } \tau \ge \abs{p}.
         \end{cases}$$
   and this implies
   \begin{align*}
      & \dint_0^{\abs{p}/2}
         \prts{\tau+1}^\lbd
         \prts{\abs{p}-\tau+1}^\nu
         \prts{\abs{p}+1}^\eps \dtau\\
      & = \dint_0^{\abs{p}/2}
            \prts{\tau+1}^\lbd
            \prts{\abs{p} - \tau+1}^{\nu+\eps}
            \pfrac{\abs{p}+1}{\abs{p}-\tau+1}^\eps \dtau\\
       & \le 2^\eps \dint_0^{\abs{p}/2}
            \prts{\tau+1}^{\lbd+\nu+\eps} \dtau\\
       & \le \dfrac{2^\eps}{\abs{\lbd+\nu+\eps+1}},
   \end{align*}
   \begin{align*}
      & \dint_{\abs{p}/2}^{\abs{p}}
         \prts{\tau+1}^\lbd
         \prts{\abs{p}-\tau+1}^\nu
         \prts{\abs{p}+1}^\eps \dtau\\
      & = \dint_{\abs{p}/2}^{\abs{p}}
            \prts{\tau+1}^{\lbd+\eps}
            \prts{\abs{p} - \tau+1}^\nu
            \pfrac{\abs{p}+1}{\tau+1}^\eps \dtau\\
       & \le 2^\eps \dint_{\abs{p}/2}^{\abs{p}}
            \prts{\abs{p}-\tau+1}^{\lbd+\nu+\eps} \dtau\\
       & \le \dfrac{2^\eps}{\abs{\lbd+\nu+\eps+1}}
   \end{align*}
   and
   \begin{align*}
      \dint_{\absp}^{+\infty}
         \prts{\tau+1}^\lbd
         \prts{\tau-\absp+1}^\nu
         \prts{\absp+1}^\eps \dtau
      & \le \dint_{\absp}^{+\infty}
         \prts{\tau+1}^{\lbd+\eps}
         \prts{\tau-\absp+1}^\nu\dtau\\
      & \le \dint_{\absp}^{+\infty}
         \prts{\tau-\absp+1}^{\lbd+\nu+\eps}\dtau\\
       & \le \dfrac{1}{\abs{\lbd+\nu+\eps+1}}.
   \end{align*}
   Hence, if $p < 0$ we have
      $$ \dint_0^{+\infty}
            \prts{\tau+1}^\lbd
            \prts{\abs{\tau+p}+1}^\nu
            \prts{\abs{p}+1}^\eps \dtau
         \le \dfrac{2^{\eps+1}+1}{\abs{\lbd+\eps+\nu+1}}.$$
\end{proof}

\begin{proof}[Proof of Theorem~$\ref{thm:global:alpha_t,s = (mu(t)-mu(s)+1)^a(|mu(s)|+a)^eps...}$]
   Since for $r \le s$ we have
      $$ \beta^+_{s,r} \alpha_{r,s}
         =  D^2 \prts{\mu(s)-\mu(r)+1}^{c+d+\eps}
            \pfrac{\abs{\mu(r)}+1}{\mu(s)-\mu(r) +1}^\eps \prts{\abs{\mu(s)}+1}^\eps$$
   and for $r \ge s$ we have
      $$ \beta^-_{s,r} \alpha_{r,s}
         = D^2 \prts{\mu(r)-\mu(s)+1}^{a+b+\eps}
            \pfrac{\abs{\mu(r)}+1}{\mu(r)-\mu(s)+1}^\eps
            \prts{\abs{\mu(s)}+1}^\eps,$$
   it follows that $c+d+\eps<0$ and $a+b+\eps <0$ are equivalent to
      $$ \lim_{r \to -\infty} \beta^+_{s,r} \alpha_{r,s}
         = \lim_{r \to +\infty} \beta^-_{s,r} \alpha_{r,s} = 0.$$

   Since for every $t \ge r \ge s$ and every $t \le r \le s$ we have
   \begin{align*}
      \prtsr{\mu(t)-\mu(r)+1} \prtsr{\mu(r)-\mu(s)+1}
       & = \prtsr{\mu(t)-\mu(r)} \prtsr{\mu(r)-\mu(s)} + \mu(t)- \mu(s) + 1\\
       & \ge \mu(t)- \mu(s) + 1,
   \end{align*}
   from $a \le 0$ it follows that
   \begin{equation*}
       \dfrac{\alpha_{t,r} \alpha_{r,s}}{\alpha_{t,s}}
         = D \dfrac{\prts{\mu(t)-\mu(r)+1}^a\prts{\mu(r)-\mu(s)+1}^a}
            {\prts{\mu(t)-\mu(s)+1}^a}\prts{\abs{\mu(r)}+1}^\eps
         \le D \prts{\abs{\mu(r)}+1}^\eps
   \end{equation*}
   for $t\ge r\ge s$ and from $c \le 0$ we also have
   \begin{equation*}
       \dfrac{\alpha_{t,r} \alpha_{r,s}}{\alpha_{t,s}}
         = \dfrac{\prts{\mu(r)-\mu(t)+1}^c\prts{\mu(s)-\mu(r)+1}^c}
            {\prts{\mu(s)-\mu(t)+1}^c}\prts{\abs{\mu(r)}+1}^\eps
         \le D \prts{\abs{\mu(r)}+1}^\eps
   \end{equation*}
   for $t\le r\le s$. Then, since $\gamma > \eps + 1$, it follows that
   \begin{align*}
      \sgm
      & = \sup_{(t,s) \in \R^2} \abs{\dint_s^t
         \dfrac{\alpha_{t,r} \Lip(f_r) \alpha_{r,s}}{\alpha_{t,s}} \dr}\\
      & \le D \delta \sup_{(t,s) \in \R^2}
         \abs{\dint_s^t \mu'(r) \prts{\abs{\mu(r)}+1}^{\eps - \gamma} \dr}\\
      & = D \delta
         \int_{-\infty}^{+\infty} \mu'(r) \prts{\abs{\mu(r)}+1}^{\eps - \gamma} \dr\\
      & = D \delta
         \int_{-\infty}^{+\infty} \prts{\abs{\tau}+1}^{\eps - \gamma} \dtau\\
      & = \dfrac{2 D \delta}{\abs{\eps-\gamma+1}}.
   \end{align*}
   Here we made the substitution $\tau=\mu(r)$.

   Making the substitution $\tau=\mu(s) - \mu(r)$ we have
   \begin{align*}
      & \dint_{-\infty}^s \beta^+_{s,r} \Lipfr \alpha_{r,s} \dr\\
      & \le D^2 \delta  \dint_{-\infty}^s \mu'(r)
         \prts{\mu(s)-\mu(r)+1}^{c+d}
         \prts{\abs{\mu(r)}+1}^{\eps-\gamma}
         \prts{\abs{\mu(s)}+1}^\eps \dr\\
      & = D^2 \delta \dint_0^{+\infty}
         \prts{\tau+1}^{c+d}
         \prts{\abs{\tau-\mu(s)}+1}^{\eps-\gamma}
         \prts{\abs{\mu(s)}+1}^\eps \dtau
   \end{align*}
   and with the substitution $\tau = \mu(r)-\mu(s)$ we obtain
   \begin{align*}
      & \dint_s^{+\infty} \beta^-_{s,r} \Lipfr \alpha_{r,s} \dr\\
      & \le D^2 \delta  \dint_s^{+\infty}
         \mu'(r) \prts{\mu(r)-\mu(s)+1}^{a+b}
         \prts{\abs{\mu(r)}+1}^{\eps-\gamma}
         \prts{\abs{\mu(s)}+1}^\eps \dr\\
      & = D^2 \delta \dint_0^{+\infty}
         \prts{\tau+1}^{a+b}
         \prts{\abs{\tau+\mu(s)}+1}^{\eps-\gamma}
         \prts{\abs{\mu(s)}+1}^\eps \dtau.
   \end{align*}
   Using Lemma~\ref{lemma:int_0^+inf (1+t)^lbd (|t+p|+1)^nu...<=...} it follows that
   \begin{align*}
      \omg
      & = \sup_{s \in R} \prtsr{
         \dint_{-\infty}^s \beta^+_{s,r} \Lipfr \alpha_{r,s} \dr
         + \dint_s^{+\infty} \beta^-_{s,r} \Lipfr \alpha_{r,s} \dr}\\
      & \le D^2 \delta \prts{\dfrac{2^{\eps+1}+1}{\abs{\maxs{a+b,c+d}+2\eps-\gamma+1}}
         + \dfrac{1}{\abs{\mins{a+b,c+d}+2\eps-\gamma+1}}}
   \end{align*}
   Hence, if $\delta$ is sufficiently small we have $2\sgm+2\omg<1$ and the theorem is proved.
\end{proof}

In the next corollary we will consider $\mu(t)= t$, i.e.,
\begin{equation*}
   \begin{split}
      \alpha_{t,s}
      & =
         \begin{cases}
            D (t-s+1)^a (|s|+1)^\eps\\
            D (s-t+1)^c (|s|+1)^\eps\\
         \end{cases}\\
      \beta^+_{t,s}
         & = D (t-s+1)^d (|s|+1)^\eps\\
      \beta^-_{t,s}
         & = D (s-t+1)^b (|s|+1)^\eps\\
   \end{split}
   \begin{split}
      \text{ for all } (t,s) \in \R^2_\ge,\\
      \text{ for all } (t,s) \in \R^2_\le,\\
      \text{ for all } (t,s) \in \R^2_\ge,\\
      \text{ for all } (t,s) \in \R^2_\le,
   \end{split}
\end{equation*}
and this type of trichotomies are called \textit{nonuniform polynomial trichotomies}.

\begin{corollary}
   Let $X$ be a Banach space. Suppose that equation~\eqref{eq:v'=A(t)v} admits a nonuniform polynomial trichotomy and let $f \colon \R \times X \to X$ be a continuous function such that~\eqref{eq:f(t,0)=0} and~\eqref{eq:f is Lip} are satisfied and
      $$ \Lip (f_r)\le \delta (|r|+1)^{-\gamma}$$
   with $\gamma > 0$. If~\eqref{ine:a+b+eps<0|c+d+eps<0} is satisfied, $\delta$ is sufficiently small and
      $$ a, c \le 0, \ \ \
         2\eps \le \gamma \ \ \ \text{ and } \ \ \
         \eps-\gamma + 1 < 0,$$
   then~\eqref{eq:v'=A(t)v+f(t,v)} admits a global Lipschtiz invariant center manifold, i.e., there is $N \in \ ]0,1[$ and a unique $\phi \in \cX_N$ such that
   \begin{equation*}
      \Psi_\tau(\cV_\phi)\subset \cV_\phi \ \ \text{ for every } \tau \in \R
   \end{equation*}
    and
   \begin{equation*}
      \norm{\Psi_{t-s}(p_{s,\xi}) - \Psi_{t-s}(p_{s,\bxi})}
         \le
         \begin{cases}
            \dfrac{DN}{\omg} \, \prts{t-s+1}^a \prts{\abs{s}+1}^\eps \, \norm{\xi-\bxi}
               \ \ \text{ if } t \ge s,\\[3mm]
            \dfrac{DN}{\omg} \, \prts{s-t+1}^c \prts{\abs{s}+1}^\eps \, \norm{\xi-\bxi}
               \ \ \text{ if } t \le s,
         \end{cases}
   \end{equation*}
   for all $(s,\xi), (s,\bxi) \in G$ and where $p_{s,\xi} = \prts{s,\xi,\phi(s,\xi)}$ and $p_{s,\bxi} = \prts{s,\bxi,\phi(s,\bxi)}$.
\end{corollary}

\section{Proof of the main theorem}\label{sec:proof:thm|global}

Before doing the proof of Theorem~\ref{thm:global} we need to prove some lemmas.

\begin{lemma} \label{lema:M&N}
   If $\sgm$ and $\omg$ are positive numbers satisfying
      $$2 \sgm + 2 \omg<1 ,$$
   then there exist $M \in\ ]1,2[$ and $N \in\ ]0,1[$ such that
   \begin{equation}\label{def:M+N}
      \sgm = \dfrac{M-1}{M(1+N)}
      \ \ \ \text{ and } \ \ \
      \omg = \dfrac{N}{M(1+N)}.
   \end{equation}
\end{lemma}

\begin{proof}
   Clearly, equalities~\eqref{def:M+N} are equivalent to
   \begin{equation}\label{eq:(M-1)/sgm=N/omg=M(1+N)}
      \dfrac{M-1}{\sgm}
      = \dfrac{N}{\omg}
      = M(1+N).
   \end{equation}
   Hence making
   \begin{equation}\label{def:M}
      M
      = \dfrac{1-\sgm+\omg - \sqrt{1-2\sgm-2\omg + (\sgm-\omg)^2}}{2\omg}
   \end{equation}
   and
   \begin{equation}\label{def:N}
      N
      = \dfrac{1-\sgm-\omg - \sqrt{1-2\sgm-2\omg + (\sgm-\omg)^2}}{2\sgm}
   \end{equation}
   we obtain immediately the first equality in~\eqref{eq:(M-1)/sgm=N/omg=M(1+N)}. Taking into account that
   \begin{equation*}
      \begin{aligned}
         (1-\sgm+\omg)^2 - \prtsr{1-2\sgm-2\omg+(\sgm-\omg)^2} = 4\omg,\ \\
         (1-\sgm-\omg)^2 - \prtsr{1-2\sgm-2\omg+(\sgm-\omg)^2} = 4\sgm\omg
      \end{aligned}
   \end{equation*}
   and $\sgm+\omg < 1 - \sgm - \omg$,  we have
   \begin{equation*}
      M
      = \dfrac{4\omg}
         {2\omg\prtsr{1-\sgm+\omg + \sqrt{1-2\sgm-2\omg+(\sgm-\omg)^2}}}
      < \dfrac{2}{1-\sgm+\omg+|\sgm-\omg|}
      \le 2
   \end{equation*}
   and
   \begin{align*}
      N
      & = \dfrac{4\sgm\omg}
         {2\sgm \prts{1-\sgm-\omg + \sqrt{1-2\sgm-2\omg+(\sgm-\omg)^2}}}\\
      & < \dfrac{2 \omg}{1-\sgm-\omg + |\sgm-\omg|}\\
      & < \dfrac{2 \omg}{\sgm+\omg + |\sgm-\omg|}\\
      & \le 1.
   \end{align*}
   Moreover, using the definition of $N$ and $M$ we can put
   \begin{align*}
      1 + \dfrac{1}{N}
      & = \dfrac{1-\sgm+\omg+\sqrt{1-2\sgm-2\omg+(\sgm-\omg)^2}}{2\omg}
      = \dfrac{1}{\omg M}
   \end{align*}
   and this proves the second equality in~\eqref{eq:(M-1)/sgm=N/omg=M(1+N)}. To finish the proof we note it is clear that $N > 0$ and since $M = 1 + \sgm N/\omg$ we have $M > 1$.
\end{proof}

From now on the numbers $M$ and $N$ will be given by~\eqref{def:M} and~\eqref{def:N}. Moreover, the number $N$ mentioned in Theorem~\ref{thm:global} is also given by~\eqref{def:N}.

It is easy to see that $\cX_N$ is a complete metric space with the metric
\begin{equation} \label{def:d(phi,psi)}
   d(\phi,\psi)
   = \sups{\frac{\norm{\phi(t,\xi) - \psi(t,\xi)}}{\norm{\xi}}
      \colon (t,\xi) \in G , \ \xi \ne 0} ,
\end{equation}
for all $\phi, \psi \in \cX_N$.

Making
\begin{equation*}
   G'
   = \bigcup_{s \in \R} \R \times \set{s} \times E_s
   = \R \times G   ,
\end{equation*}
let $\cB $ be the set of all continuous functions $x \colon G' \to X$ such that
\begin{align}
   & x(t,s,0)=0 \ \text{ for all } (t,s) \in \R^2,
      \label{eq:x(t,s,0)=0}\\
   & x(s,s,\xi)= \xi \ \text{ for all } (s,\xi) \in G,
      \label{eq:x(s,s,xi)=xi}\\
   & x(t,s,\xi) \in E_t \ \text{ for all } (t,s,\xi) \in G'
      \label{eq:x(t,s,xi):in:E_t},\\
   & \sups{\dfrac{\norm{x(t,s,\xi)-x(t,s,\bxi)}}
      {\alpha_{t,s} \, \norm{\xi-\bxi}}
      \colon (t,s,\xi), (t,s,\bxi) \in G', \ \xi \ne \bxi} \le M
      \label{eq:Lip[x(t,s,.)]<=M_alpha_t,s}
\end{align}
where $M$ is given by~\eqref{def:M}.

From~\eqref{eq:Lip[x(t,s,.)]<=M_alpha_t,s}, it follows that
\begin{equation}\label{eq:x3}
   \norm{x(t,s,\xi) - x(t,s,\bxi)}
   \le M \alpha_{t,s} \norm{\xi-\bxi}
      \text{ \ for all } (t,s,\xi), (t,s,\bxi) \in G'
\end{equation}
and making $\bxi = 0$ in~\eqref{eq:x3}, from~\eqref{eq:x(t,s,0)=0} we have
\begin{equation}\label{eq:x4}
   \norm{x(t,s,\xi)}
   \le M \alpha_{t,s} \norm{\xi}
      \text{ \ for all } (t,s,\xi) \in G'.
\end{equation}
Defining, for every $x, y \in \cB$,
\begin{equation}\label{def:d'(x,y)}
   d'(x,y) = \sups{
      \dfrac{\norm{x(t,s,\xi)-y(t,s,\xi)}}{\alpha_{t,s} \, \norm{\xi}}
      \colon (t,s,\xi) \in G', \ \xi \ne 0},
\end{equation}
it is easy to see that $\prts{\cB,d'}$ is a complete metric space.

Another complete metric space that we are going to need is $\cC = \cB \times \cX_N$ equipped with the metric defined by
\begin{equation*}
   d''\prts{(x,\phi),(y,\psi)} = d'(x,y) + d(\phi,\psi) \text{ for all } (x,\phi), (y,\psi) \in \cC.
\end{equation*}

To prove~\eqref{eq:psi_tau(V)_C=_V} we need to prove that there is a solution of~\eqref{eq:v'=A(t)v+f(t,v)} in the form
\begin{equation*}
   \prts{x(t,s,\xi),\phi^+(t,x(t,s,\xi)), \phi^-(t,x(t,s,\xi))} \in E_t \times F^+_t \times F^-_t
\end{equation*}
for some $(x, \phi) \in \cC$. From~\eqref{eq:x(t)=T_t,sP_s...},~\eqref{eq:y^+(t)=T_t,sQ^+_s...} and~\eqref{eq:y^-(t)=T_t,sQ^-_s...}, this is equivalent to solving the equations
\begin{align}
   & x(t,s,\xi)
      = T_{t,s}P_s\xi
         + \dint_s^t T_{t,r}P_r f(r,x(r,s,\xi),\phi(r,x(r,s,\xi))) \dr,
         \label{eq:x(t,s,xi)=Txi+integral}\\
   & \phi^+ (t,x(t,s,\xi))
      = T_{t,s}Q^+_s \phi^+(s,\xi)
         + \dint_s^t T_{t,r}Q^+_r f(r,x(r,s,\xi),\phi(r,x(r,s,\xi))) \dr,
         \label{eq:phi1= T eta+integral}\\
   & \phi^- (t,x(t,s,\xi))
      = T_{t,s}Q^-_s \phi^-(s,\xi)
         + \dint_s^t T_{t,r}Q^-_r f(r,x(r,s,\xi),\phi(r,x(r,s,\xi))) \dr.
         \label{eq:phi2= T teta +integral}
\end{align}

Next we state a lemma where we obtain an equivalent form of equations~\eqref{eq:phi1= T eta+integral} and~\eqref{eq:phi2= T teta +integral}.

\begin{lemma}   \label{lema da equival}
   Let $\prts{x,\phi} \in \cC$ such that~\eqref{eq:x(t,s,xi)=Txi+integral} is satisfied. The following properties are equivalent:
   \begin{enumerate}[\lb=$\alph*)$,\lm=6mm]
      \item for every $(t,s,\xi) \in G'$, equations~\eqref{eq:phi1= T eta+integral} and~\eqref{eq:phi2= T teta +integral} hold;
      \item for every $(s,\xi) \in G$,
         \begin{equation}\label{eq:phi1=integral de -oo a s}
            \phi^+ (s,\xi)
            = \int_{-\infty}^s
               T_{s,r} Q^+_r f(r, x(r,s,\xi),\phi( r,x(r,s,\xi)))\dr
         \end{equation}
         and
         \begin{equation}\label{eq:phi2=-integral de s a +oo}
            \phi^- (s,\xi)
            = -\int_s^{+\infty}
               T_{s,r} Q^-_r f(r, x(r,s,\xi),\phi( r,x(r,s,\xi)))\dr.
         \end{equation}
   \end{enumerate}
\end{lemma}

\begin{proof}
   First we must prove that the integrals in~\eqref{eq:phi1=integral de -oo a s} and~\eqref{eq:phi2=-integral de s a +oo} are convergent. From~\eqref{eq:||f(t,x)||<=...},~\eqref{eq:||phi(t,xi)||<=...} and~\eqref{eq:x4} we obtain
   \begin{equation}\label{||f(r,x,phi)||<=...}
      \begin{split}
         \norm{f(r, x(r,s,\xi),\phi( r,x(r,s,\xi)))}
         & \le \Lip(f_r) \norm{x(r,s,\xi)+\phi( r,x(r,s,\xi))}\\
         & \le \Lip(f_r) \prts{\norm{x(r,s,\xi)}+\norm{\phi( r,x(r,s,\xi))}}\\
         & \le \Lip(f_r) \prts{\norm{x(r,s,\xi)}+ N \norm{x(r,s,\xi)}}\\
         & \le M(1+N) \Lip(f_r) \alpha_{r,s} \norm{\xi}
      \end{split}
   \end{equation}
   for every $r,s \in \R$ and using~\ref{eq:trich:D2},~\eqref{||f(r,x,phi)||<=...},~\eqref{def:omg} and~\ref{eq:trich:D3} we have
   \begin{align*}
      & \int_{-\infty}^s
         \norm{T_{s,r} Q^+_r f(r, x(r,s,\xi),\phi( r,x(r,s,\xi)))}\dr\\
      & \le \int_{-\infty}^s \norm{T_{s,r} Q^+_r}
                \norm{f(r, x(r,s,\xi),\phi( r,x(r,s,\xi)))}\dr \\
      & \le M(1+N) \norm{\xi} \int_{-\infty}^s \beta^+_{s,r} \Lip(f_r) \alpha_{r,s} \dr\\
      & \le M(1+N) \omg \norm{\xi}
   \end{align*}
   and
   \begin{align*}
      & \int_s^{+\infty} \norm{T_{s,r} Q^-_r
                f(r, x(r,s,\xi),\phi( r,x(r,s,\xi)))}\dr\\
      & \le \int_s^{+\infty} \norm{T_{s,r} Q^-_r}
         \norm{f(r, x(r,s,\xi),\phi( r,x(r,s,\xi)))} \dr \\
      & \le M(1+N) \norm{\xi}
         \int_s^{+\infty} \beta^-_{s,r} \Lip(f_r) \alpha_{r,s}\dr\\
     & \le M(1+N) \omg \norm{\xi}
   \end{align*}
   for every $(s, \xi) \in G$.  Thus the integrals~\eqref{eq:phi1=integral de -oo a s} and~\eqref{eq:phi2=-integral de s a +oo} are convergent.

   Now we prove that $a) \Rightarrow b)$. Suppose that~\eqref{eq:phi1= T eta+integral} and~\eqref{eq:phi2= T teta +integral} hold for all $(t,s,\xi) \in G'$. Then, from~\eqref{eq:phi1= T eta+integral} we have
   \begin{align*}
      \phi^+ (s,\xi)
      & = T_{s,t} \phi^+(t,x(t,s,\xi))
         - \int^t_s T_{s,t} T_{t,r} Q^+_r
            f(r, x(r,s,\xi), \phi( r,x(r,s,\xi)))\dr\\
      & = T_{s,t} Q^+_t \phi(t,x(t,s,\xi))
            - \int^t_s T_{s,r} Q^+_r f(r, x(r,s,\xi), \phi(r,x(r,s,\xi)))\dr.
   \end{align*}
   Since by~\ref{eq:trich:D2},~\eqref{eq:||phi(t,xi)||<=...} and~\eqref{eq:x4} we have
   \begin{align*}
      \norm{T_{s,t}Q^+_t \phi(t,x(t,s, \xi))}
      & \le \beta^+_{s,t} \norm{\phi(t,x(t,s, \xi))}\\
      & \le N \beta^+_{s,t} \norm{x(t,s,\xi)}\\
      & \le MN \norm{\xi} \beta^+_{s,t} \alpha_{t,s}~,
   \end{align*}
   by~\eqref{eq:lim+oo b^- a = lim-oo b^+ a =0}, making $t \to -\infty$, we conclude that
      $$ \lim_{t \to -\infty} T_{s,t} Q^+_t \phi(t,x(t, s,\xi))
         = 0$$
   and this implies
   \begin{align*}
      \phi^+ (s,\xi)
      & = - \int_s^{-\infty}
         T_{s,r} Q^+_r f(r, x(r,s,\xi), \phi( r,x(r,s,\xi)))\dr \\
      & = \int_{-\infty}^s
         T_{s,r} Q^+_r f(r, x(r,s,\xi), \phi( r,x(r,s,\xi)))\dr,
   \end{align*}
   i.e.,~\eqref{eq:phi1=integral de -oo a s} holds.

   Similarly, from~\eqref{eq:phi2= T teta +integral} we have
   \begin{align*}
      \phi^-(s,\xi)
      & =  T_{s,t} \phi^-(t, x(t,s, \xi)) - \int^t_s
         T_{s,t} T_{t,r} Q^-_r f(r, x(r,s,\xi),\phi(r,x(r,s,\xi)))\dr\\
      & = T_{s,t} Q^-_t \phi(t, x(t,s, \xi)) - \int^t_s
         T_{s,t} T_{t,r} Q^-_r f(r, x(r,s,\xi),\phi(r,x(r,s,\xi)))\dr
   \end{align*}
   and by~\ref{eq:trich:D3},~\eqref{eq:||phi(t,xi)||<=...} and~\eqref{eq:x4} it follows that
   \begin{align*}
      \norm{T_{s,t} Q^-_t \phi(t, x(t,s, \xi))}
      & \le \beta^-_{s,t} \norm{\phi(t, x(t,s, \xi))}\\
      & \le N \beta^-_{s,t} \norm{x(t,s, \xi)}\\
      & \le MN \norm{\xi} \beta^-_{s,t} \alpha_{t,s}.
   \end{align*}
   Letting $t \to + \infty$ we have
      $$ \lim_{t \to +\infty} T_{s,t} Q^-_t \phi(t,x(t, s,\xi))= 0$$
   and we obtain
   \begin{align*}
       \phi^- (s,\xi)
       =  -  \int_s^{+\infty} T_{s,r}  Q^-_r
               f(r, x(r,s,\xi),\phi( r,x(r,s,\xi))) \dr
   \end{align*}
   for every $(s, \xi) \in G$. Hence $a) \Rightarrow b)$.

   Now we will prove that $b) \Rightarrow a)$. Assuming that for every $(s, \xi) \in G$ identities~\eqref{eq:phi1=integral de -oo a s} and~\eqref{eq:phi2=-integral de s a +oo} hold, applying $T_{t,s}$ to both sides of equation~\eqref{eq:phi1=integral de -oo a s} we have
      $$ T_{t,s} \phi^+ (s,\xi)
         = \int_{-\infty}^s
            T_{t,r}  Q^+_r f(r, x(r,s,\xi), \phi( r,x(r,s,\xi)))\dr$$
   and this implies
   \begin{align*}
         & T_{t,s} \phi^+ (s,\xi) + \int_s^t T_{t,r}  Q^+_r
            f(r, x(r,s,\xi), \phi( r,x(r,s,\xi)))\dr \\
         & = \int_{-\infty}^t T_{t,r}  Q^+_r f(r, x(r,s,\xi),
            \phi( r,x(r,s,\xi)))\dr \\
         & = \int_{-\infty}^t T_{t,r}  Q^+_r f(r, x(r,t,x(t,s,\xi)),
            \phi( r,x(r,t,x(t,s,\xi))))\dr \\
         & = \phi^+ ( t, x(t,s,\xi))
   \end{align*}
   for every $\prts{t, s, \xi} \in G'$. In a similar way we have
      $$ T_{t,s} \phi^- (s,\xi)
         = -\int_s^{+\infty}
            T_{t,s} T_{s,r}  Q^-_r f(r, x(r,s,\xi), \phi( r,x(r,s,\xi)))\dr$$
   and thus
   \begin{align*}
         & T_{t,s} \phi^- (s,\xi) + \int_s^t T_{t,r}  Q^-_r
            f(r, x(r,s,\xi), \phi( r,x(r,s,\xi)))\dr\\
         & = - \int_t^{+\infty} T_{t,r}  Q^-_r f(r, x(r,s,\xi),
            \phi( r,x(r,s,\xi)))\dr \\
         & = - \int_t^{+\infty} T_{t,r}  Q^-_r f\prts{r, x(r,t,x(t,s,\xi)),
            \phi\prts{r,x\prts{r,t,x\prts{t,s,\xi}}}}\dr \\
         & = \phi^- ( t, x(t,s,\xi))
   \end{align*}
   for every $(t, s , \xi) \in G' $. Therefore $b) \Rightarrow a)$ and this completes the proof of the lemma.
\end{proof}

Consider in $\cC$ the operator $J$ such that, to each $(x,\phi) \in \cC$, assigns a function $J(x,\phi) \colon G' \to X $ defined, for every $\prts{t,s,\xi} \in G'$,  by
\begin{equation*}
   \prtsr{J (x, \phi)}(t,s,\xi)
   = T_{t,s}P_s\xi
      + \dint_s^t T_{t,r}P_r f(r,x(r,s,\xi),\phi(r,x(r,s,\xi))) \dr.
\end{equation*}

\begin{lemma}
   For every $\prts{x,\phi} \in \cC$, we have
      $$ J(x,\phi) \in \cB.$$
\end{lemma}

\begin{proof}
   Given $(x, \phi) \in \cC$, from~\eqref{eq:x(t,s,0)=0},~\eqref{eq:phi(t,0)=0} and~\eqref{eq:f(t,0)=0} it follows immediately that $J(x, \phi)(t,s,0)=0$ for every $(t,s) \in \R^2$ and by definition we have $J(x,\phi)(s,s,\xi)= \xi$ for all $(s, \xi) \in G$ and $J(x,\phi)(t,s,\xi) \in E_t$ for all $(t,s,\xi) \in G'$.

   Moreover, for all $(t, s, \xi), (t,s,\bxi) \in G'$ and using~\ref{eq:trich:D1} it follows that
   \begin{align*}
      \norm{J(x,\phi)(t,s,\xi)-J(x,\phi)(t,s,\bxi)}
      & \le \norm{T_{t,s} P_s} \norm{\xi - \bxi}
         + \dint_s^t \norm{T_{t,r}P_r} \,\cdot\, \gamma_{r,s,\xi,\bxi} \dr\\
      & \le \alpha_{t,s} \norm{\xi - \bxi}
         + \dint_s^t \alpha_{t,r} \,\cdot\, \gamma_{r,s,\xi,\bxi} \ \dr,
   \end{align*}
   where
   \begin{align*}
      \gamma_{r,s,\xi,\bxi}
      := \norm{f(r,x(r,s,\xi),\phi(r,x(r,s,\xi)))
         - f(r,x(r,s,\bxi),\phi(r,x(r,s,\bxi)))}.
   \end{align*}
   From~\eqref{eq:||f(t,x)-f(t,y)||<=...},~\eqref{eq:||phi(t,xi)-phi(t,bxi)||<=...} and~\eqref{eq:x3} we have
   \begin{equation}
      \begin{split} \label{eq:gama r,s}
         & \gamma_{r,s,\xi,\bxi}\\
         & \le \Lip(f_r) \prts{\norm{x(r,s,\xi) - x(r,s,\bxi)}
            + \norm{\phi(r,x(r,s,\xi))-\phi(r,s,x(r,s,\bxi)) }}\\
         & \le  \Lip(f_r) \prts{\norm{x(r,s,\xi) - x(r,s,\bxi)}
            + N \norm{x(r,s,\xi))-x(r,s,\bxi)) }}\\
         & = (1+N) \Lip(f_r) \norm{x(r,s,\xi) - x(r,s,\bxi)}\\
         & \le  M(1+N) \norm{\xi-\bxi} \Lip(f_r) \alpha_{r,s}
      \end{split}
   \end{equation}
   and so by~\eqref{def:sgm} and Lemma~\ref{lema:M&N} it follows that
   \begin{align*}
      & \norm{J(x,\phi)(t,s,\xi) - J(x,\phi)(t,s,\bxi)}\\
      & \le \alpha_{t,s} \norm{\xi -\bxi}+ M(1+N)\norm{\xi -\bxi} \dint_s^t \alpha_{t,r}
           \Lip(f_r)  \alpha_{r,s} \dr\\
      & \le \prts{1 + M(1+N)\sgm} \alpha_{t,s} \norm{ \xi-\bxi }\\
      & = M \alpha_{t,s} \norm{\xi -\bxi}
   \end{align*}
   for every $(x,\phi) \in \cC$ and every $(t,s,\xi), (t,s,\bxi) \in G'$.
   Then considering $\xi \ne \bxi$ we have
      $$ \dfrac{ \norm{J(x,\phi)(t,s,\xi)-J(x,\phi)(t,s,\bxi)}}
            {\alpha_{t,s}\norm{ \xi-\bxi }}
         \le M.$$
   Hence $J(x,\phi)$ satisfies~\eqref{eq:x(t,s,0)=0},~\eqref{eq:x(s,s,xi)=xi},~\eqref{eq:x(t,s,xi):in:E_t} and~\eqref{eq:Lip[x(t,s,.)]<=M_alpha_t,s} and this proves that $J(x,\phi) \in \cB$ for all $(x,\phi) \in \cC$.
\end{proof}

In $\cC$ define the operator $L$ that assigns to every $(x,\phi) \in \cC$ a function
   $$ L(x,\phi) \colon G \to X   $$
defined, for every $(s,\xi) \in G$, by
   $$ \prtsr{L(x,\phi)}(s,\xi)
      = \prtsr{L^+(x,\phi)}(s,\xi) + \prtsr{L^-(x,\phi)}(s,\xi),$$
where
\begin{equation*}
      \prtsr{L^+(x, \phi)} (s, \xi)
      = \dint_{-\infty}^s
         T_{s,r} Q^+_r f(r, x(r,s,\xi), \phi(r,x(r,s,\xi))) \dr
\end{equation*}
and
\begin{equation*}
   \prtsr{L^-(x, \phi)} (s, \xi)
   = -\dint_s^{+\infty} T_{s,r} Q^-_r f(r, x(r,s,\xi), \phi(r,x(r,s,\xi)))\dr.
\end{equation*}

\begin{lemma}
   For every $\prts{x,\phi} \in \cC$, we have
      $$ L(x,\phi) \in \cX_N.$$
\end{lemma}

\begin{proof}
   From~\eqref{eq:x(t,s,0)=0},~\eqref{eq:phi(t,0)=0} and~\eqref{eq:f(t,0)=0} it follows that $\prtsr{L(x,\phi)}(t,0)= 0$. Moreover, by definition we have $\prtsr{L(x,\phi)}(t,\xi) \in F^+_t \oplus F^-_t$.

   From~\ref{eq:trich:D2},~\ref{eq:trich:D3} and~\eqref{eq:gama r,s} it follows for every $(s,\xi), (s,\bxi) \in G$ that
   \begin{align*}
      \| L^+ (x, \phi) (s, \xi) - L^+(x, \phi) ( s, \bxi) \|
      & \le \dint_{-\infty}^s \|T_{s,r} Q^+_r \| \ \gamma_{r,s,\xi,\bxi} \ \dr\\
      &  \le  M(1+N) \norm{\xi-\bxi}
         \dint_{-\infty}^s \beta^+_{s,r} \Lip(f_r) \alpha_{r,s}  \dr
   \end{align*}
   and
   \begin{align*}
      \| L^- (x, \phi) (s, \xi) - L^-(x, \phi) ( s, \bxi) \|
      & \le \dint_s^{+\infty} \|T_{s,r} Q^-_r\| \ \gamma_{r,s,\xi,\bxi} \ \dr\\
      & \le M(1+N) \norm{\xi-\bxi} \dint_s^{+\infty} \beta^-_{s,r} \Lip(f_r) \alpha_{r,s}\dr
   \end{align*}
   and so, using~\eqref{def:omg} and Lemma~\ref{lema:M&N} we have
   \begin{align*}
      & \norm{\prtsr{L(x,\phi)}(s,\xi) -\prtsr{L(x,\phi)}(s,\bxi)}\\
      & \le \norm{\prtsr{L^+(x,\phi)}(s,\xi) -\prtsr{L^+(x,\phi)}(s,\bxi)}
         + \norm{\prtsr{L^-(x,\phi)}(s,\xi) -\prtsr{L^-(x,\phi)}(s,\bxi)}\\
      & \le M(1+N)\norm{\xi-\bxi}
         \prts{\dint_{-\infty}^s \beta^+_{s,r} \Lipfr \alpha_{r,s} \dr
         + \dint_s^{+\infty} \beta^-_{s,r} \Lipfr \alpha_{r,s} \dr}\\
      & \le M(1+N)\omg\norm{\xi-\bxi}\\
      & = N \norm{\xi-\bxi},
   \end{align*}
   and the proof is complete.
\end{proof}

\begin{lemma}
   For every $(x,\phi), (y,\psi) \in \cC$ we have
   \begin{equation}\label{eq:d'(J(x,phi),J(y,psi))<=...}
      d'\prts{J(x,\phi), J(y,\psi)}
         \le  \sgm  \prtsr{(1+N) d'(x,y)  + M  d(\phi,\psi)}
   \end{equation}
   and
   \begin{equation}\label{d(L(x,phi),L(y,psi))<=...}
      d \prts{L(x, \phi),L(y,\psi)}
      \le \omg \prtsr{(1+N) d'(x,y)  + M d(\phi,\psi)}.
   \end{equation}
\end{lemma}

\begin{proof}
   For every $(r,s,\xi) \in G'$, putting
      $$ \overline{\gamma}_{r,s,\xi}
         := \| f(r,x(r,s,\xi), \phi(r,x(r,s,\xi))) -
            f(r,y(r,s,\xi), \psi(r,y(r,s,\xi))) \|,  $$
   and using~\ref{eq:trich:D1} we have
      $$ \|J(x,\phi)(t,s, \xi)- J(y,\psi)(t,s,\xi) \|
         \le \abs{\dint_s^t \|T_{t,r} P_r\| \
            \overline{\gamma}_{r,s,\xi} \dr}
         \le \abs{\dint_s^t \alpha_{t,r} \ \overline{\gamma}_{r,s,\xi} \dr}.$$
   By~\eqref{eq:||f(t,x)-f(t,y)||<=...},~\eqref{eq:||phi(t,xi)-phi(t,bxi)||<=...},~\eqref{def:d'(x,y)}, ~\eqref{def:d(phi,psi)} and~\eqref{eq:x4} we obtain
   \begin{equation}\label{eq:f(r,x,phi)-f(r,y,psi)}
      \begin{split}
         \overline{\gamma}_{r,s,\xi}
          & = \| f(r,x(r,s,\xi), \phi(r,x(r,s,\xi))) -
                  f(r,y(r,s,\xi), \psi(r,y(r,s,\xi))) \| \\
          & \le \Lip(f_r) \prtsr{\| x(r,s,\xi) - y(r,s,\xi) \| +
                \| \phi( r,x(r,s,\xi))  - \psi(r,y(r,s,\xi)) \|}\\
          & \le \Lip(f_r) [\| x(r,s,\xi) - y(r,s,\xi) \|
               + \| \phi( r,x(r,s,\xi))  - \phi( r,y(r,s,\xi))\|\\
            & \qquad  +
                \|\phi( r,y(r,s,\xi))- \psi(r,y(r,s,\xi)) \|]\\
          & \le \Lip(f_r) \prtsr{(1+N) \norm{x(r,s,\xi) - y(r,s,\xi)}
            + d( \phi,\psi) \|y(r,s,\xi)) \|}\\
          & \le \Lip(f_r) \prtsr{d'(x,y) \alpha_{r,s} \norm{\xi} (1+N) +
                d( \phi,\psi) M \alpha_{r,s} \norm{\xi} }\\
          & \le \Lip(f_r) \alpha_{r,s} \norm{\xi} \prtsr{(1+N) d'(x,y)  +
               M  d( \phi,\psi)},
      \end{split}
   \end{equation}
   it follows by~\eqref{def:sgm} that
   \begin{align*}
      & \|J(x,\phi)(t, s, \xi)- J(y,\psi)(t, s, \xi) \|\\
      & \le \dint_s^t \alpha_{t,r} \Lip(f_r) \alpha_{r,s} \dr \
         \norm{\xi} \prtsr{(1+N) d'(x,y)  + M  d(\phi,\psi)}\\
      & \le \alpha_{t,s} \sgm \norm{\xi}
         \prtsr{(1+N) d'(x,y)  + M  d(\phi,\psi)},
   \end{align*}
   for every $(t,s,\xi) \in G'$. Thus from~\eqref{def:d'(x,y)} we get~\eqref{eq:d'(J(x,phi),J(y,psi))<=...}.

   On the other hand, using again~\ref{eq:trich:D2} and~\eqref{eq:f(r,x,phi)-f(r,y,psi)} we have
   \begin{align*}
      & \|L^+(x,\phi)(s, \xi)- L^+(y,\psi)(s,\xi)\|\\
      & \le \dint_{-\infty}^s
         \|T_{s,r} Q^+_r\| \ \overline{\gamma}_{r,s,\xi}\dr\\
      & \le \norm{\xi} \prtsr{(1+N) d'(x,y)  + M  d(\phi,\psi)}
         \dint_{-\infty}^s \beta^+_{s,r} \Lip(f_r) \alpha_{r,s} \dr
   \end{align*}
   and also from~\ref{eq:trich:D3} and~\eqref{eq:f(r,x,phi)-f(r,y,psi)} it follows that
   \begin{align*}
      & \|L^-(x,\phi)(s, \xi)- L^-(s,\psi)(t,\xi)\| \\
      & \le \dint_s^{+\infty}
         \|T_{s,r} Q^-_r\| \ \overline{\gamma}_{r,s,\xi}\dr\\
      & \le \norm{\xi} \prtsr{(1+N) d'(x,y)  + M  d( \phi,\psi)}
         \dint_s^{+\infty} \beta^-_{s,r} \Lip(f_r) \alpha_{r,s} \dr
   \end{align*}
   and thus from~\eqref{def:omg} we obtain
      $$ \|L(x,\phi)(s, \xi)- L(s,\psi)(t,\xi)\|
         \le \norm{\xi} \omg \prtsr{(1+N) d'(x,y)  + M  d( \phi,\psi)}.$$
   Therefore, from~\eqref{def:d(phi,psi)} we get~\eqref{d(L(x,phi),L(y,psi))<=...}.
\end{proof}

Now we define an operator on $\cC$ and we will prove that it is a contraction
and this will be essential in the proof of Theorem~$\ref{thm:global}$.

Define the operator $T \colon \cC \to \cC $ by
\begin{equation*}
   T(x, \phi)
   = \prts{J(x, \phi), L (x,\phi)}
   = \prts{J (x, \phi), L^+ (x,\phi), L^- (x, \phi)}.
\end{equation*}

\begin{lemma} \label{T:contraction}
   The operator $T \colon \cC \to \cC$ is a contraction.
\end{lemma}

\begin{proof}
   Let $(x,\phi), (y,\psi) \in \cC$. From last lemma it follows that
   \begin{align*}
      d''\prts{T(x,\phi),T(y,\psi)}
      & = d''\prts{(J(x,\phi),L(x, \phi)), ( J(y, \psi),L(y,\psi)}\\
      & = d'\prts{J(x,\phi),J(y, \psi)} + d \prts{L(x, \phi),L(y,\psi)}\\
      & \le (\sgm+\omg) \prts{(1+N) d'(x,y)  + M  d( \phi,\psi)}\\
      & \le \prts{\sgm + \omg} \maxs{1+N,M} d''((x,\phi),(y,\psi))
   \end{align*}
   and, since $\sgm+\omg<1/2$, $N<1$ and $M<2$, we obtain
   \begin{equation*}
      \prts{\sgm + \omg} \maxs{1+N,M} < 1
   \end{equation*}
   and this implies that $T$ is a contraction.
\end{proof}

Now we are going to prove Theorem~$\ref{thm:global}$.

\begin{proof}[Proof of Theorem~$\ref{thm:global}$]
   By Lemma~\ref{T:contraction} the operator $T$ is a contraction. Since $\cC$ is a complete metric space, by Banach fixed point theorem, there is a unique point $(x,\phi) \in \cC$ such that
      $$ T(x,\phi) = (x,\phi)$$
   and this fixed point verifies~\eqref{eq:x(t,s,xi)=Txi+integral},~\eqref{eq:phi1=integral de -oo a s} and~\eqref{eq:phi2=-integral de s a +oo}. By Lemma~\ref{lema da equival} we know that solve the last two equations is equivalent to solve~\eqref{eq:phi1= T eta+integral} and~\eqref{eq:phi2= T teta +integral}, if~\eqref{eq:x(t,s,xi)=Txi+integral} holds. Therefore, from~\eqref{eq:x(t)=T_t,sP_s...},~\eqref{eq:y^+(t)=T_t,sQ^+_s...} and~\eqref{eq:y^-(t)=T_t,sQ^-_s...}, we conclude the existence of the invariant manifold, i.e., the existence of a unique
   \begin{equation*}
     \phi=(\phi^+,\phi^-) \in \cX_N \ \text{ such that }\ \Psi_\tau ( \cV_\phi)\subseteq \cV_\phi
   \end{equation*}
   for every $\tau \in \R$. Moreover, for every $t \in \R$ and every $(s,\xi),  (s,\bxi) \in G$ we have
   \begin{align*}
      & \norm{\Psi_{t-s}(s,\xi,\phi(s,\xi)) - \Psi_{t-s}(s,\bxi,\phi(s, \bxi))}\\
      & = \norm{(t, x(t, s, \xi), \phi(t, x(t, s, \xi))) -
         (t, x(t, s, \bxi), \phi(t, x(t, s, \bxi)))} \\
      & \le \norm{x(t, s, \xi)- x(t, s, \bxi)} +
         \norm{\phi(t, x(t, s, \xi)) - \phi(t, x(t, s, \bxi))}\\
      & \le (1+N) \norm{x(t, s, \xi)- x(t, s, \bxi)} \\
      & \le M(1+N) \alpha_{t,s} \norm{\xi- \bxi}\\
      & = \dfrac{N}{\omg} \alpha_{t,s} \norm{\xi- \bxi}
   \end{align*}
   and this completes the proof.
\end{proof}
\bibliographystyle{elsart-num-sort}

\end{document}